\documentclass[11pt]{article}
\usepackage{graphicx}
\usepackage{amsmath,amsfonts,amsthm}
\usepackage{amssymb,latexsym}
\usepackage{fixmath}
\usepackage{mathrsfs,amsbsy}
\usepackage{dsfont}
\usepackage{enumerate}

\def\wtd{\widetilde}
\def\what{\widehat}

\DeclareMathOperator*{\argmin}{argmin}

\DeclareMathOperator{\diag}{diag}

\DeclareMathOperator{\rank}{rank}

\DeclareMathOperator{\F}{F}
\DeclareMathOperator{\HH}{H}
\DeclareMathOperator{\T}{T}

\def\ba{\pmb{a}}
\def\bb{\pmb{b}}
\def\bc{\pmb{c}}

\def\be{\pmb{e}}
\def\Bf{\pmb{f}}
\def\bg{\pmb{g}}
\def\bh{\pmb{h}}

\def\bl{\pmb{l}}
\def\bm{\pmb{m}}

\def\bp{\pmb{p}}
\def\bq{\pmb{q}}
\def\br{\pmb{r}}

\def\bt{\pmb{t}}

\def\bw{\pmb{w}}
\def\bx{\pmb{x}}
\def\by{\pmb{y}}
\def\bz{\pmb{z}}

\newtheorem{proposition}{Proposition}[section]
\newtheorem{theorem}{Theorem}[section]
\newtheorem{lemma}{Lemma}[section]
\newtheorem{corollary}{Corollary}[section]

\theoremstyle{definition}

\newtheorem{example}{Example}[section]

\numberwithin{equation}{section}
\numberwithin{figure}{section}
\numberwithin{table}{section}

\allowdisplaybreaks

\usepackage{kbordermatrix}

\usepackage{caption}
\usepackage{amsfonts,amsthm}
\usepackage{amsmath,amssymb,enumerate,color}
\usepackage{algorithm,algorithmic}
\usepackage{epsfig,graphicx,psfrag,mathrsfs,subfigure}
\usepackage{eucal}
\usepackage{yhmath}
\usepackage{diagbox}
\RequirePackage{multirow}
 \usepackage{cleveref}
\usepackage{textcomp}
\usepackage[]{fontenc}
\usepackage{extarrows}
\def\wtd{\widetilde}
\def\what{\widehat}
\usepackage{rotating}
\usepackage{lscape}
\usepackage{bm}
\usepackage{xurl}
\usepackage{xspace}
\usepackage{tikz}
\usetikzlibrary{petri} 
\usetikzlibrary{shadows,arrows,positioning,shapes.geometric} 

\usepackage{geometry}
\def\ba{\pmb{a}}
\def\bb{\pmb{b}}
\def\bc{\pmb{c}}

\def\be{\pmb{e}}
\def\Bf{\pmb{f}}
\def\BF{{\pmb{F}}}
\def\BW{{\pmb{W}_{\otimes}}}

\def\bg{\pmb{g}}
\def\bh{\pmb{h}}

\def\bl{\pmb{l}}

\def\bp{\pmb{p}}
\def\bq{\pmb{q}}
\def\br{\pmb{r}}

\def\bt{\pmb{t}}

\def\bw{\pmb{w}}
\def\bx{\pmb{x}}
\def\by{\pmb{y}}
\def\bz{\pmb{z}}

\def\diag{{\rm diag}}

\def\scrR{\mathscr{R}}

\def\wtd{\widetilde}
\def\what{\widehat}

\def\bbC{\mathbb{C}}

\def\bbP{\mathbb{P}}
\def\bbR{\mathbb{R}}

\renewcommand{\algorithmicrequire}{\textbf{Input:}}
\renewcommand{\algorithmicensure}{\textbf{Output:}}

\numberwithin{equation}{section}
\numberwithin{figure}{section}
\numberwithin{table}{section}

\graphicspath{{figures/}{figure/}{pictures/}%
	{picture/}{pic/}{pics/}{image/}{images/}}

\title{Rational minimax approximation of matrix-valued functions}

\author{Lei-Hong Zhang\thanks{Corresponding author. School of Mathematical Sciences, Soochow University, Suzhou 215006, Jiangsu, China. This work was
 supported in part by the National Natural Science Foundation of China (NSFC-12471356,  NSFC-12371380), Jiangsu Shuangchuang Project (JSSCTD202209), Academic Degree and Postgraduate Education Reform Project of Jiangsu Province, and China Association of Higher Education under grant 23SX0403. 
        Email: {\tt longzlh@suda.edu.cn}.} \and Ya-Nan Zhang\thanks{School of Mathematical Sciences, Soochow University, Suzhou 215006, Jiangsu, China. This work was
 supported in part by the National Natural Science Foundation of China NSFC-12171319. Email: {\tt ynzhang@suda.edu.cn}.} \and Chenkun Zhang\thanks{School of Mathematical Sciences, Soochow University, Suzhou 215006, Jiangsu, China. Email: {\tt 19816896524@163.com}.} \and Shanheng Han\thanks{School of Mathematical Sciences, Soochow University, Suzhou 215006, Jiangsu, China.   Email: {\tt 3468805603@qq.com}.}}
 
 \date{ }

\begin{document}

\maketitle

\begin{abstract}
In this paper, we present a rigorous framework for rational minimax approximation of matrix-valued functions that generalizes classical scalar approximation theory. Given sampled data $\{(x_\ell, {F}(x_\ell))\}_{\ell=1}^m$ where ${F}:\mathbb{C} \to \mathbb{C}^{s \times t}$ is a matrix-valued function, we study the problem of finding a matrix-valued rational approximant  ${R}(x) = {P}(x)/q(x)$ (with ${P}:\mathbb{C} \to \mathbb{C}^{s \times t}$ a matrix-valued polynomial and $q(x)$ a nonzero scalar polynomial of prescribed degrees)  that minimizes the worst-case Frobenius norm error over the given nodes:
$$
\inf_{{R}(x) = {P}(x)/q(x)} \max_{1 \leq \ell \leq m} \|{F}(x_\ell) - {R}(x_\ell)\|_{\F}.
$$
By reformulating this min-max optimization problem through Lagrangian duality, we derive a maximization dual problem over the probability simplex.  We analyze weak and strong duality properties and establish a sufficient condition ensuring that the solution of the dual problem yields the minimax approximant $R(x)$. For numerical implementation, we propose an efficient method (\textsf{m-d-Lawson}) to solve the dual problem, generalizing Lawson's iteration to matrix-valued functions. Convergence analysis of \textsf{m-d-Lawson} is established. Numerical experiments are conducted and compared to state-of-the-art approaches, demonstrating its efficiency as a novel computational framework for matrix-valued rational approximation.
\end{abstract}


\medskip
{\small
\medskip

} 

\tableofcontents
 \clearpage

\section{Introduction}\label{sec_intro}
Let $F:\bbC\rightarrow \bbC^{s\times t}$ be a matrix-valued function defined on a subset of $\Omega$ in the complex plane.  For a given   set of distinct nodes  ${\cal X}=\{x_{\ell}\}_{{\ell}=1}^m$ in $\Omega$, suppose that we have the  discrete sampled data  $\{(x_{\ell},F(x_{\ell}))\}_{{\ell}=1}^m$ with 
\begin{equation}\label{eq:Ft}
F(x_{\ell})=\left[\begin{array}{ccc}f_{11}(x_{\ell}) &  \cdots & f_{1t}(x_{\ell})   \\\ \vdots& \ddots & \vdots  \\ f_{s1}(x_{\ell}) &  \cdots & f_{st}(x_{\ell})\end{array}\right]\in \bbC^{s\times t}.
\end{equation}
Let $\|F(x_{\ell})\|_{\F}$ be the Frobenius norm of $F(x_{\ell})$ and  $\bbP_{n}$ be the set of complex polynomials with degree less than or equal to $n$. In this paper, we consider the following discrete  rational  minimax approximation for the matrix-valued function $F(x)$:
\begin{equation}\label{eq:bestf0}
\eta_{\infty}:=\inf_{R \in\scrR_{(s,t)}}\max_{1\le \ell \le m}\|F(x_{\ell})-R(x_{\ell})\|_{\F}^2,
\end{equation}
where  
\begin{equation}\label{eq:rats}
\scrR_{(s,t)}:=\left\{R(x)=\left[\begin{array}{ccc}r_{11}(x) &  \cdots & r_{1t}(x)   \\\ \vdots& \ddots & \vdots  \\ r_{s1}(x) &  \cdots & r_{st}(x)\end{array}\right]  \Big| r_{ij}(x)=\frac{p_{ij}(x)}{q(x)}, p_{ij}\in \bbP_{n_{ij}},~0\not\equiv q\in\bbP_{d}\right\},
\end{equation}
$n_{ij}$ and $d$ are prescribed integers. The $(i,j)$th function $\frac{p_{ij}(x)}{q(x)}$ of $R(x)$ is said to be a rational function of type $(n_{ij},d)$ where $p_{ij}\in \bbP_{n_{ij}}$ is the $(i,j)$ entry of a matrix-valued polynomial $P:\bbC\rightarrow \bbC^{s\times t}$.  Moreover, we assume  w.l.g.  throughout the paper that
\begin{equation}\label{eq:numberm}
m\ge  \max_{1\le i\le s,1\le j\le t }\{n_{ij}+d+2\},
\end{equation}
since otherwise, $R \in \scrR_{(s, t)}$ simplifies to an interpolation problem for $\{(x_\ell, F(x_\ell))\}_{\ell=1}^m$.

It is clear that \eqref{eq:bestf0} generalizes the classical rational minimax approximation for complex-valued functions (e.g., \cite{elli:1978, gutk:1983, rutt:1985, sava:1977, this:1993, tref:2019a, will:1972, will:1979, wulb:1980, zhha:2025, zhyy:2025}), reducing to the classical case when $s = t = 1$. For clarity, we hereafter refer to \eqref{eq:bestf0} as the matrix-valued (or block) rational minimax approximation.
Following the setting of matrix-valued rational approximation \cite{begu:2015, begu:2017,gogu:2021, gust:2009, limp:2022,zhzy:2025}, we enforce a common denominator $q \in \bbP_d$ for all entries $r_{ij}(x)$ in \eqref{eq:bestf0}. Notably, the matrix-valued rational minimax approximation arises in diverse applications, including   the linear time-invariant system (e.g., \cite{begu:2017,gogu:2021}), frequency-domain multiport rational modeling (see e.g., \cite{demd:2009,gogu:2021,gust:2006,gust:2009,guse:1999}), computer-aided designing of the microwave duplexer \cite{trai:2010,trms:2007,zhzy:2025} and nonlinear eigenvalue problems \cite{guti:2017,gupt:2022,limp:2022,saem:2020}.

 In the literature, several numerical methods exist for rational  approximation of matrix-valued functions. Particularly, Gosea and Güttel  \cite{gogu:2021} investigated a range of efficient algorithms, including the  rational Krylov fitting (RKFIT) method \cite{begu:2017},  set-valued AAA \cite{limp:2022}, Loewner framework interpolation approach \cite{anli:2017,gova:2022,maan:2007}, and the vector fitting (VF) approach in \cite{drgb:2015,gust:2009}; additionally, they introduced a block-AAA extension of the AAA algorithm \cite{nase:2018}  and conducted extensive numerical experiments to evaluate these methods.  
 To approximate   three rational functions emerging in microwave duplexer filter design \cite{trai:2010,trms:2007}, a recent study \cite{zhzy:2025}  proposes a vector-valued AAA-type approach ({\sf v-AAA-Lawson}), building upon the AAA algorithm \cite{nase:2018} and AAA-Lawson method \cite{fint:2018}; {\sf v-AAA-Lawson} is also applicable to general matrix-valued rational approximations.

 However, we observe that most existing algorithms lack a rigorously defined optimization objective, leading to approximations with no guaranteed optimality. As an extension of   least squares, the absolute root mean squared error (RMSE)
\begin{equation}\label{eq:RMSE}
{\tt RMSE}=  \sqrt{\frac1m\sum_{\ell=1}^m\sum_{i=1}^s\sum_{j=1}^t\left|f_{ij}(x_\ell)-r_{ij}(x_\ell)\right|^2} = \sqrt{\frac1m\sum_{\ell=1}^m\Big\|F(x_{\ell})-R(x_{\ell})\Big\|_{\F}^2}
\end{equation}
or the relative one is commonly adopted as an accuracy measure for a computed approximation $R:\bbC\rightarrow \bbC^{s\times t}$ \cite{gogu:2021}. However, there is no general algorithm that guarantees convergence to the globally optimal solution with minimal RMSE (see \cite[A.5]{begu:2017}).
 
{\it Contributions}.
In this paper, with a rigorous and clear objective \eqref{eq:bestf0} in the minimax sense, we study the problem of finding the associated matrix-valued rational minimax approximant. Exploring  recent advances in duality-based treatments for the minimax approximation \cite{yazz:2023,zhha:2025,zhyy:2025}, we develop a numerical method to solve \eqref{eq:bestf0}. Our key contributions are as follows.
\begin{itemize}
\item[1).]  {Theoretical contributions:}
\begin{itemize}
\item reformulation of the minimax problem \eqref{eq:bestf0} into a single-level minimization via linearization, along with a detailed analysis of their theoretical connections;
\item derivation of a maximization dual problem of the linearization through Lagrange duality. We establish weak and strong duality between the linearized problem and its dual, and provide a sufficient condition (Theorem \ref{thm:strongdualityRuttan}) ensuring that the solution of the dual problem yields the best approximant for \eqref{eq:bestf0}.
\end{itemize}
\item[2).]  {Computational contributions:}
\begin{itemize}
\item an extension, namely {\sf m-d-Lawson}, of the \textsf{d-Lawson} iteration \cite{zhyy:2025} to solve \eqref{eq:bestf0} via the dual problem;
\item convergence analysis of {\sf m-d-Lawson} is established;
\item numerical validation demonstrating the efficiency of {\sf m-d-Lawson} against state-of-the-art approaches, qualifying it as a novel computational framework for matrix-valued rational approximation.
\end{itemize}
\end{itemize} 
 
%

 {\it Paper organization}. The structure of the remaining paper is summarized below.  Section \ref{sec_dual} introduces a linearization of \eqref{eq:bestf0} and discusses its relation with \eqref{eq:bestf0};  the Lagrangian duality is applied to this linearization, yielding a maximization dual problem. In Section \ref{sec_d2}, we focus on the implementation details of computing the dual function at a given dual variable; particularly, the Vandermonde with Arnoldi process \cite{brnt:2021} is introduced in this computation step. Section \ref{sec_strongduality} is devoted to strong duality between the linearization of \eqref{eq:bestf0} and the dual problem. Under strong duality, in Section \ref{sec:complement}, we develop a complementary slackness property at the maximizer of the dual problem and discuss its connection with the {\it extreme points} (aka {\it reference points}) of the best approximant of \eqref{eq:bestf0}. The new method  {\sf m-d-Lawson} is proposed in Section \ref{sec_lawson}, which generalizes the classical Lawson's iteration \cite{laws:1961} to the matrix-valued functions. The convergence analysis of {\sf m-d-Lawson} is presented in Section \ref{sec:convergence}. Section \ref{sec:numerical} provides numerical experiments comparing {\sf m-d-Lawson} with state-of-the-art methods for matrix-valued rational approximation, followed by detailed performance analysis. Finally, Section \ref{sec:conclusion} summarizes the key findings and conclusions.\\

{\it Notation}. 
Throughout this paper, we employ the following notation:
\begin{itemize}

 \item  ${\mathbb C}^{n\times m}$ (resp. ${\mathbb R}^{n\times m}$) is the set
of all $n\times m$ complex (resp. real) matrices and $\bullet^{\HH}$ (resp. $\bullet^{\T}$) represents the conjugate transpose (resp. transpose).


\item 
All real or complex numbers (or functions) are denoted by Latin (e.g., $a$) or Greek letters (e.g., $\alpha$). Column vectors are represented by boldface lowercase Latin (e.g., $\ba$) or Greek letters (e.g., $\bm{\alpha}$), where boldface lowercase symbols with subscripts (e.g., $\bp_{ij}$ and $\boldsymbol{\tau}_i$) also denote column vectors. Matrices are denoted by boldface or non-boldface uppercase Latin (e.g., $A$ and $\mathbf{A}$) or Greek letters (e.g., $\bm{\Phi}$); similarly, uppercase letters with subscripts (e.g., $F_{ij}$ and $\bm{\Psi}_{ij}$) also represent matrices. For simplicity, we sometimes express column vectors and matrices in terms of their elements, as $\ba = [a_1,\ldots,a_n]^{\T}$ and $A = [a_{ij}]$, respectively.

\item   $I_n\equiv [\be_1,\be_2,\dots,\be_n]\in\bbR^{n\times n}$ represents the $n$-by-$ n$ identity matrix, where $\be_i$ is its $i$th column with $1\le i\le n$. The vector $\mathbf{1}_g = [1,\dots,1]^{\T} \in \bbR^g$ is the $g$-dimensional all-ones column vector.

\item 
 ${\tt i}=\sqrt{-1}$ represents the imaginary unit, and for $\mu\in \bbC$, we write $\mu= \mu^{\tt r}+ {\tt i} \mu^{\tt i}$ and $\bar\mu= \mu^{\tt r}- {\tt i} \mu^{\tt i}$ where ${\rm Re }(\mu)=\mu^{\tt r}\in \bbR,~{\rm Im }(\mu)=\mu^{\tt i}\in \bbR$ and $|\mu|=\sqrt{(\mu^{\tt r})^2+(\mu^{\tt i})^2}$.

\item  $\diag(\bx)=\diag(x_1,\dots,x_n)$   denotes the diagonal matrix associated with the vector $\bx\in\bbC^n$, and   $\bx./\by=\frac{\bx.}{\by}=[x_1/y_1,\dots,x_n/y_n]$ for two vectors $\bx,\by\in \bbC^n$ with $y_j\ne 0,~1\le j\le n$. 

\item  ${\rm span}(A)$ is a subspace spanned by columns of $A$; the  $k$th Krylov subspace generated by a matrix $A$ on $\bb$ is denoted by
$
{\cal K}_k(A,\bb)={\rm span}(\bb,A\bb,\dots,A^{k-1}\bb).
$


\end{itemize}

\section{The dual problem of the matrix-valued rational minimax approximation} \label{sec_dual}
\subsection{Representation of   rational functions}\label{sbsec:representation}
In order to represent a rational function $r_{ij}(x)=\frac{p_{ij}(x)}{q(x)}$, we can choose basis functions $\{\psi_0(x),\dots,\psi_{n_{ij}}(x)\}$ and $\{\phi_0(x),\dots,\phi_{d}(x)\}$ to parameterize 
\begin{equation}\nonumber
r_{ij}(x)=\frac{p_{ij}(x)}{q(x)}=\frac{[\psi_0(x),\dots,\psi_{n_{ij}}(x)] \ba_{ij}}{[\phi_0(x),\dots,\phi_{d}(x)] \bb},~~\mbox{for ~some~} \ba_{ij}\in \bbC^{n_{ij}+1},~ \bb\in \bbC^{d+1}.
\end{equation}  
In principle, $\{\psi_0(x),\dots,\psi_{n_{ij}}(x)\}$ and $\{\phi_0(x),\dots,\phi_{d}(x)\}$ can be any basis functions for $\bbP_{n_{ij}}$ and $\bbP_{d}$, respectively. Besides the   monomial bases,  the barycentric representation for the rational functions \cite{fint:2018,gogu:2021,nase:2018} can also be a choice. While any parameterization that is able to represent the rational functions is mathematically equivalent, numerically, particular implementations and techniques should be concerned to deal with the numerical stability. In this paper, we will mainly use the monomial  basis $\psi_i(x)=\phi_i(x)=x^{i}$ and rely on the Vandermonde with Arnoldi  process \cite{brnt:2021,hoka:2020,zhsl:2023} to handle the ill-conditioned problems inherited from the Vandermonde   matrix at nodes ${\cal X}=\{x_\ell\}_{\ell=1}^m$; in particular, for $p_{ij}(x)$ in the rational  function $r_{ij}(x)=\frac{p_{ij}(x)}{q(x)}$, we denote  by 
\begin{equation}\label{eq:basisVmatrix}
{\bf \Psi}_{ij}={\bf \Psi}_{ij}(x_1,\dots,x_m;n_{ij}):=\left[\begin{array}{cccc}\psi_0(x_1) & \psi_1(x_1) & \cdots & \psi_{n_{ij}}(x_1) \\\psi_0(x_2) & \psi_1(x_2) & \cdots & \psi_{n_{ij}}(x_2)  \\ \vdots & \cdots & \ddots& \vdots  \\\psi_0(x_m) & \psi_1(x_m) & \cdots & \psi_{n_{ij}}(x_m) \end{array}\right]\in \bbC^{m\times (n_{ij}+1)}
\end{equation}  
the associated Vandermonde matrix; therefore, 
\begin{equation}\label{eq:vecpij}
\bp_{ij}:=[p_{ij}(x_{1}),\dots, p_{ij}(x_{m})]^{\T}={\bf \Psi}_{ij}\ba_{ij}\in \bbC^{m}. 
\end{equation}
Similarly, we denote by ${\bf \Phi}={\bf \Phi}(x_1,\dots,x_m;d)\in \bbC^{m\times (d+1)}$ the associated Vandermonde matrix for the denominator $q(x)$ at ${\cal X}$, whose $(i,j)$th entry is $\phi_{j-1}(x_i)$; thus  
\begin{equation}\label{eq:vecq}
\bq:=[q(x_{1}),\dots, q(x_{m})]^{\T}={\bf \Phi}\bb\in \bbC^{m}. 
\end{equation}

Given an $R(x) \in\scrR_{(s,t)}$, if the associated maximum approximation error
\begin{equation}\label{eq:exi}
e(R):=\max_{x_{\ell}\in {\cal X}}\|F(x_{\ell})-R(x_{\ell})\|_{\F}^2
\end{equation} 
is finite, then we can assume w.l.g. that $q(x_j)\ne 0,~\forall x_j\in {\cal X}$. To see this,  note that  if $q(z)=0$ for some $z \in \cal X$, then $z$ must be a zero for all the numerators $p_{ij}(x)$ with $1\le i\le s$ and $1\le j\le t$, because otherwise $e(R)=\infty$. This implies that zeros $z$ of $q(x)$ in ${\cal X}$ can be canceled in all $p_{ij}(x)/q(x)$, and the values $r_{ij}(z)$ are well-defined through cancellation of these common zeros (i.e., this induces an equivalence in $\scrR_{(s,t)}$). Due to this fact, if a solution  exists for \eqref{eq:bestf0}, there is an $R \in \scrR_{(s,t)}$ such that $q(x_j)\ne 0$ for any $x_j\in {\cal X}$.  
 
\subsection{The Lagrange dual problem and weak duality}\label{sbsec:dual}
Following  \cite{zhyy:2025}, a key observation for addressing \eqref{eq:bestf0} is   the following equivalence:
$$\max_{1\le \ell \le m}\|F(x_{\ell})-R(x_{\ell})\|_{\F}^2\le \eta \Longleftrightarrow\|F(x_{\ell})-R(x_{\ell})\|_{\F}^2\le \eta,  ~\forall 1\le \ell\le m.
$$
Providing $q(x_{\ell})\ne 0 ~(\forall 1\le \ell\le m)$, the latter $m$ constraints can further  be  represented as 
\begin{equation}\label{eq:mconsts}
\sum_{i=1}^s \sum_{j=1}^t \left|f_{ij}(x_\ell)q(x_{\ell})-p_{ij}(x_{\ell})\right|^2\le \eta |q(x_{\ell})|^2, \quad 1\le \ell\le m.
\end{equation}
Thus, similar to the scalar version in \cite{zhyy:2025}, we can construct the following minimization  
 \begin{align}\nonumber
\eta_2:=&\inf_{\eta\in \bbR,~p_{ij}\in \bbP_{n_{ij}},~0\not \equiv q\in \bbP_{d}}\eta \\\label{eq:linearity}
 s.t., ~&\sum_{i=1}^s \sum_{j=1}^t \left|f_{ij}(x_\ell)q(x_{\ell})-p_{ij}(x_{\ell})\right|^2\le \eta |q(x_{\ell})|^2, \quad 1\le \ell\le m.
\end{align}

In Theorem \ref{thm:attainable_equi}, we prove that the infimum $\eta_2$ of \eqref{eq:linearity} is indeed attainable, qualifying this linearization as a standard optimization. This is a parallel generalization of \cite[Theorem 2.1]{zhha:2025} whose proof can be suitably modified to lead the following assertion.
 \begin{theorem}\label{thm:attainable_equi}
Given $m\ge \max_{1\le i\le s,1\le j\le t }\{n_{ij}+d+2\}$ distinct nodes ${\cal X}=\{x_\ell\}_{\ell=1}^m$ in $ \bbC$, let $\eta_2$ be the infimum of \eqref{eq:linearity}. Then $\eta_2$  is attainable by certain polynomials $p_{ij}\in \bbP_{n_{ij}}~(1\le i\le s,1\le j\le t)$ and $0\not \equiv q\in \bbP_{d}$. 
\end{theorem}

\begin{proof}
For any  $0\not \equiv q\in \bbP_{d}$, as $m\ge d+2$, there exists $x_j\in \cal{X}$ with $q(x_j)\ne 0$. This implies that ${\eta_2}\ge 0.$ Moreover, $\eta_2<\infty$  as we can choose $p_{ij}\equiv 0~(1\le i\le s, 1\le j\le t)$ and $q\equiv 1$ to bound the left hand side of each constraint of \eqref{eq:linearity} by a constant.

Let $\{(\eta^{(k)},\{p_{ij}^{(k)}\},q^{(k)})\}_k$ be a feasible  sequence such that $p_{ij}^{(k)}\in \bbP_{n_{ij}}$, $0\not \equiv q^{(k)}\in \bbP_{d}$ and $\eta^{(k)}\rightarrow \eta_2$ as $k\rightarrow \infty$.  For any $\varsigma^{(k)}\ne 0$, since $\{(\eta^{(k)},\varsigma^{(k)}\{p_{ij}^{(k)}\},\varsigma^{(k)}q^{(k)})\}$ is also feasible for \eqref{eq:linearity},  $\varsigma^{(k)}$ can be chosen so that the coefficient vector $\bb^{(k)}=[b_0^{(k)},\dots,b_{d}^{(k)}]^{\T}\in \bbC^{d+1}$ of $\varsigma^{(k)}q^{(k)}(x)=\sum_{i=0}^{d}b_i^{(k)}x^i$ satisfies  $\|\bb^{(k)}\|_2=1 ~(\forall k\ge 0)$. Let $\tilde p_{ij}^{(k)}:=\varsigma^{(k)} p_{ij}^{(k)}$ and $\tilde q^{(k)}:=\varsigma^{(k)}q^{(k)}$. Thus $\{\tilde q^{(k)}\}_k$ has a convergent subsequence resulting from a convergent subsequence of $\{\bb^{(k)}\}_k$. For simplicity of presentation, we assume that $\{\tilde q^{(k)}\}_k$ itself converges to $0\not \equiv  q\in \bbP_{d}$, and $\eta^{(k)}\le \rho_0 ~(\forall k\ge 0)$ for some $\rho_0>0$.
 
We next show that for each pair $(i,j)$, $\{p^{(k)}_{ij}\}_k$ also contains a convergent subsequence. Indeed, by $$|f_{ij}(x_\ell)\tilde q^{(k)}(x_\ell)-\tilde p_{ij}^{(k)} (x_\ell)|^2\le \sum_{i=1}^s \sum_{j=1}^t|f_{ij}(x_\ell)\tilde q^{(k)}(x_\ell)-\tilde p_{ij}^{(k)} (x_\ell)|^2\le  {\eta^{(k)}} |\tilde q^{(k)}(x_\ell)|^2,$$
we have 
$$
|\tilde p_{ij}^{(k)} (x_\ell)|\le \sqrt{\eta^{(k)}} |\tilde q^{(k)}(x_\ell)|+ |f_{ij}(x_\ell)\tilde q^{(k)}(x_\ell)|\le  (\sqrt{\rho_0}+ {f_{\max}})\rho_1 =:\tilde \rho_1,~\forall x_\ell\in {\cal X},
$$
where $\rho_1=\sum_{i=0}^{d}|x_J|^i$ with $|x_J|=\max_{1\le j\le m}|x_j|$. Denoting $\tilde p_{ij}^{(k)}(x)=\sum_{h=0}^{n_{ij}}a_{h}^{(k)}x^h$ in the monomial basis, we can choose the first $n_{ij}+1$ nodes $\{x_\ell\}_{\ell=1}^{n_{ij}+1}$ to have 
$$[\tilde p_{ij}^{(k)}(x_1),\dots,\tilde p_{ij}^{(k)}(x_{{n_{ij}}+1})]^{\T}=V(x_1,\dots,x_{n_{ij}+1})[a_0^{(k)},\dots,a_{n_{ij}}^{(k)}]^{\T},$$
where $V(x_1,\dots,x_{n_{ij}+1})$ is the Vandermonde matrix associated with $\{x_\ell\}_{\ell=1}^{n_{ij}+1}$. This gives 
$$\|[a_0^{(k)},\dots,a_{n_{ij}}^{(k)}]\|_2\le \tilde \rho_1 \sqrt{n_{ij}+1}\left\|[V(x_1,\dots,x_{n_{ij}+1})]^{-1}\right\|_2,~\forall k\ge 0.$$ Therefore,    $\{\tilde p_{ij}^{(k)}\}_k$ also has a convergent subsequence with a limit polynomial $p_{ij}\in \bbP_{n_{ij}}$ resulting from a convergent subsequence of $\{[a_0^{(k)},\dots,a_{n_{ij}}^{(k)}]\}_k$. Consequently, we can find $\{\eta_2,\{p_{ij}\},q\}$ with $0\not \equiv q\in \bbP_d$ which solves \eqref{eq:linearity}.
\end{proof}

Following \cite{zhha:2025} again, we next provide a sufficient condition: $0\le d\le \min_{1\le i\le s,1\le j\le t }\{n_{ij}\}$ to ensure $\eta_\infty=\eta_2$. Note that such a sufficient condition includes both the matrix-valued polynomial case (i.e., $d=0$) as well as the rational case with all  rational $r_{ij}$ of type $(d,d)$.  The result, together with Theorem \ref{thm:attainable_equi}, paves the way to solve \eqref{eq:bestf0} through the linearization \eqref{eq:linearity}.

\begin{theorem}\label{thm:linearity}
Given $m\ge \max_{1\le i\le s,1\le j\le t }\{n_{ij}+d+2\}$ distinct nodes ${\cal X}=\{x_j\}_{j=1}^m$ on $\Omega\subset\bbC$, let $ {\eta_2}$ be the  infimum of \eqref{eq:linearity}. If $0\le d\le \min_{1\le i\le s,1\le j\le t }\{n_{ij}\}$, then $\eta_2= \eta_\infty$; furthermore, in this case, whenever \eqref{eq:bestf0} has a solution $R\in  \scrR_{(s, t)}$ with $q(x_j)\ne 0 ~(\forall  1\le j\le m)$,  the tuple $(\eta_\infty, \{p_{ij}\},q)$ is a solution of  \eqref{eq:linearity}. 
\end{theorem}

\begin{proof} 
The proof is adopted by a slight modification of \cite[Theorem 2.4]{zhha:2025}. 
First, ${{\eta_2}}\le \eta_\infty$ holds due to   the fact that  for sequence $\{ \eta_{\infty}^{(k)}, R^{(k)}\}_k, ~\forall R^{(k)}\in  \scrR_{(s, t)}$ with $q^{(k)}(x)\ne 0 ~(\forall x\in {\cal X}$, $(\eta_{\infty}^{(k)},\{p_{ij}^{(k)}\},q^{(k)})$ is feasible for \eqref{eq:linearity}, and thus $ {\eta_2}\le \eta_{\infty}^{(k)}\rightarrow \eta_\infty$. 

We prove  ${\eta_2}= \eta_\infty$ by ruling out the possibility ${{\eta_2}}< \eta_\infty$. Assume ${\eta_2}< \eta_\infty$. By Theorem \ref{thm:attainable_equi},  we let $(\eta_2,\{\hat p_{ij}\},\hat q)$ be a solution to \eqref{eq:linearity} where $\hat p_{ij}\in\bbP_{n_{ij}}~(1\le i\le s, 1\le j\le t)$ and  $0\not \equiv  \hat q\in \bbP_d$.  Suppose w.l.g. that  $\hat q(x_\ell)=0$ for $\ell=1,\dots,h$. As $\hat q\not \equiv 0$, we have   $h\le d.$ For any pair $(i,j)$, we consider the constraint of \eqref{eq:linearity} on nodes $x_\ell, ~\forall  \ell=1,2,\dots, h$. Noting from
 \begin{equation*}
 \sum_{i=1}^{s}\sum_{j=1}^{t} |f_{ij}(\ell)\hat q(x_\ell)-\hat p_{ij}(x_\ell)|^2 =\sum_{i=1}^{s}\sum_{j=1}^{t} |\hat p_{ij}(x_\ell)|^2\le  \eta_2 \hat q(x_\ell)=0,
 \end{equation*}
 we have $\hat p_{ij}(x_\ell)=0, \forall \ell =1,\dots,h$. 
If $h=0$ (i.e., $\hat q(x_\ell)\ne 0~\forall \ell =1,2,\dots, m$), then the proof is finished as
$$
\sum_{i=1}^s\sum_{j=1}^t|f_{ij}(x_\ell) \hat   q(x_\ell)-\hat p_{ij}(x_\ell)|^2\le  {\eta_2} |\hat q(x_\ell)|^2\Longrightarrow \sum_{i=1}^s\sum_{j=1}^t\left|f_{ij}(x_\ell) -\frac{\hat p_{ij}(x_\ell)}{\hat q(x_\ell)}\right|^2\le  {{\eta_2}}<\eta_\infty,
$$
contradicting with the fact that $\eta_\infty$ is the infimum of \eqref{eq:bestf0}. 

Thus, we next consider the case $h\ge 1$. We will show that there is a parameterized form $\hat R(x;\delta)\in  \scrR_{(s,t)}$ with its $(i,j)$th component $\hat r_{ij}(x;\delta)$  given by
$$
\hat r_{ij}(x;\delta)=\frac{\hat p_{ij}(x;\delta)}{\hat q(x;\delta)}:=\frac{\hat p_{ij}(x)+\delta \cdot p_{ij}(x)}{\hat q(x)+\delta \cdot q(x)} 
$$
for some polynomials $p_{ij}\in \bbP_{n_{ij}} ~(1\le i\le s, 1\le j\le t)$ and $q\in \bbP_{d}$ so that 
\begin{equation}\label{eq:contraction}
\max_{1\le  \ell\le m}\sum_{i=1}^s\sum_{j=1}^t\left|f_{ij}(x_\ell)-\frac{\hat p_{ij}(x_\ell;\delta)}{\hat q(x_j;\delta)}\right|^2<\eta_\infty.
\end{equation}
This implies that $\eta_\infty$ is not the infimum of \eqref{eq:bestf0}, a contradiction. 

To this end, choose $q\equiv 1$.  Since $h\le d\le \min_{1\le i\le s,1\le j\le t }\{n_{ij}\}$, there are  interpolation polynomials $p_{ij}\in \bbP_{n_{ij}}$   fulfilling  conditions $p_{ij}(x_\ell)=f_{ij}(x_\ell) ~(1\le \ell\le h)$; hence $\left|f_{ij}(x_\ell)-\frac{p_{ij}(x_\ell)}{q (x_\ell)}\right|=0,~1\le \ell \le h$, leading to 
$$\sum_{i=1}^s \sum_{j=1}^t\left|f_{ij}(x_\ell)-\frac{p_{ij}(x_\ell)}{q (x_\ell)}\right|^2=0$$ and $q (x_\ell)\ne 0$,  $\forall \ell=1,\dots, h$. 

On the other hand, for any sufficiently small $\delta\ne 0$, we know that $\hat q(x_\ell;\delta)=\hat q(x_\ell) +\delta q(x_\ell)=\hat q(x_\ell) +\delta  \ne 0,~1\le \ell\le m$. Moreover, the parameterized $\frac{\hat p_{ij}(x;\delta)}{\hat q(x;\delta)}$   satisfies
\begin{align*}
\sum_{i=1}^s\sum_{j=1}^t\left|\frac{\hat p_{ij}(x_\ell;\delta)}{\hat q(x_\ell;\delta)}-f_{ij}(x_\ell)\right|^2&=\sum_{i=1}^s\sum_{j=1}^t\left|\frac{\hat p_{ij}(x_\ell)+\delta \cdot p_{ij}(x_\ell)}{\hat q(x_\ell)+\delta \cdot q(x_\ell)}-f_{ij}(x_\ell)\right|^2\\
&=\sum_{i=1}^s\sum_{j=1}^t\left|\frac{p_{ij}(x_\ell)}{q(x_\ell)}-f_{ij}(x_\ell)\right|^2=0< \eta_\infty,~\forall  \ell=1,\dots,h.
\end{align*}
Furthermore,  $\forall \ell=h+1,\dots,m$, we have 
$$
\sum_{i=1}^s\sum_{j=1}^t |f_{ij}(x_\ell) \hat q(x_\ell)-\hat p_{ij}(x_\ell)|^2\le  {\eta_2} |\hat q(x_\ell)|^2\Longrightarrow \sum_{i=1}^s\sum_{j=1}^t\left|f_{ij}(x_\ell) -\frac{\hat p_{ij}(x_\ell)}{\hat q(x_\ell)}\right|^2\le  {\eta_2}<\eta_\infty,
$$
and for each pair $(i,j)$,
$$
\left|\frac{\hat p_{ij}(x_\ell)+\delta \cdot p_{ij}(x_\ell)}{\hat q(x_\ell)+\delta \cdot q(x_\ell)}-\frac{\hat p_{ij}(x_\ell)}{\hat q(x_\ell)}\right|=|\delta|\cdot \left|\frac{ p_{ij}(x_\ell)\hat q(x_\ell)- \hat p_{ij}(x_\ell)}{(\hat q(x_\ell)+\delta )\hat q(x_\ell)}\right|\rightarrow 0,~~{\rm as}~~\delta\rightarrow 0.
$$
Consequently,  for any sufficiently small $\delta$,  it follows 
\begin{align*}
\sum_{i=1}^s\sum_{j=1}^t\left|\frac{\hat p_{ij}(x_\ell;\delta)}{\hat q(x_\ell;\delta)}-f_{ij}(x_\ell)\right|^2&\le\sum_{i=1}^s\sum_{j=1}^t\left( \left|\frac{\hat p_{ij}(x_\ell;\delta)}{\hat q(x_\ell;\delta)}-\frac{\hat p_{ij}(x_\ell)}{\hat q(x_\ell)}\right|+\left|f_{ij}(x_\ell) -\frac{\hat p_{ij}(x_\ell)}{\hat q(x_\ell)}\right|\right)^2\\
&<\eta_\infty, ~\forall \ell=h+1,\dots, m.
\end{align*} 
This leads to \eqref{eq:contraction}, and we have $ {\eta_2}= \eta_\infty$ by contradiction.  

For the last part, according to $ {\eta_2}= \eta_\infty$ and Theorem \ref{thm:attainable_equi},    
whenever \eqref{eq:bestf0} has a solution $R\in  \scrR_{(s, t)}$ with $q(x_j)\ne 0 ~(1\le j \le m)$,  the tuple $(\eta_\infty, \{p_{ij}\},q)$ is a solution of  \eqref{eq:linearity}. 
\end{proof} 
 
Note that while \eqref{eq:bestf0} and its linearized counterpart \eqref{eq:linearity} share the same infimum $\eta_{\infty}$, this value may not be attainable in \eqref{eq:bestf0}---a fundamental feature of rational minimax approximation. The following example extends \cite[Example 2.1]{zhha:2025} to the matrix-valued setting, revealing the intrinsic distinction between \eqref{eq:bestf0} and \eqref{eq:linearity}. 
\begin{example}
  Consider \(s=t=2\), \(n_{ij}=d=1,\ \forall \ 1 \le i,j\le 2\) and $m=4$ with \(x_{\ell}=\frac{\ell-1}{4}\) (\(1\le \ell\le 4\)). Let  $F(x_1)=I_2$ and $F(x_\ell)=0_{2 \times 2}$ for $2\le \ell\le 4$.
  It is true that for \eqref{eq:linearity}, the infimum \(\eta_2=0\) is attainable by \(p_{ij}(x) \equiv 0,\ \forall i,j\) and \(q(x)=x\).  For \eqref{eq:bestf0},  consider the sequence \(R^{(k)}(x)=\left[r^{(k)}_{ij}(x)\right] \in \mathscr{R}_{(2,2)}\) with $r^{(k)}_{ij}(x)\equiv 0$  for $1\le i\ne j\le 2$, while  $r^{(k)}_{ij}(x)=\frac{1}{1+kx}$ for $1\le i=j\le 2$.
  It holds that
  \begin{equation}\nonumber
    \max_{1 \le \ell \le 4}\|F(x_{\ell})-R^{(k)}(x_{\ell})\|_{\F}^2=\frac{2}{\left(1+\frac{k}{4}\right)^2} \rightarrow 0,\ \ {\rm as}\ k \rightarrow \infty,
  \end{equation}
  which implies \(\eta_2=\eta_{\infty}\).  
  However, the infimum \(\eta_{\infty}=0\) is unattainable in \eqref{eq:bestf0} as any nonconstant  type (1,1) rational function can only take each value once. 
\end{example}

Theorem \ref{thm:attainable_equi} reveals that \eqref{eq:linearity} is a specific minimization problem, for which we can apply the traditional Lagrange duality theory (see e.g., \cite{boyd:2004,nowr:2006}). Write $p_{ij}(x_{\ell})$ (resp.,  $q(x_{\ell})$), i.e., the evaluation of  $p_{ij}(x)$ (resp.,  $q(x)$) at $x_{\ell}$,  as $p_{ij}(x_{\ell};\ba_{ij})$ (reps., $q(x_{\ell};\bb)$), in order to indicate the dependence on the coefficient vector $\ba_{ij}=\ba_{ij}^{\tt r}+ {\tt i}\ba_{ij}^{\tt i}$  (resp., $\bb=\bb^{\tt r}+ {\tt i}\bb^{\tt i}$). Now, introduce the Lagrange multipliers $0\le \bw=[w_1,\dots,w_m]^{\T}\in \bbR^m$ and define the Lagrange function of \eqref{eq:linearity} as
\begin{align}\nonumber
\lambda(\eta,\{\ba_{ij}\},\bb;\bw)&=\eta- \sum_{\ell=1}^mw_\ell\left(\eta |q(x_\ell;\bb)|^2-\sum_{i,j}\Big|f_{ij}(x_{\ell})q(x_\ell;\bb)-p_{ij}(x_\ell;\ba_{ij})\Big|^2 \right)\\\nonumber
&=\eta\left(1-\sum_{\ell=1}^mw_\ell|q(x_\ell;\bb)|^2\right)+\sum_{\ell=1}^mw_\ell \left(\sum_{i,j}\Big|f_{ij}(x_{\ell})q(x_\ell;\bb)-p_{ij}(x_\ell;\ba_{ij})\Big|^2\right).
\end{align}
It is easy to see that for a given $\bw\ge 0$, 
\begin{align}\nonumber
&\inf_{\eta,\{\ba_{ij}\},\bb}\lambda(\eta,\{\ba_{ij}\},\bb;\bw)\\
=&\left\{\begin{array}{lc}-\infty,& \mbox{if~} 1\ne \sum_{\ell=1}^mw_\ell|q(x_\ell;\bb)|^2, \\\inf_{\{\ba_{ij}\},\bb}\sum_{\ell=1}^mw_\ell \left(\sum_{i,j}\big|f_{ij}(x_\ell)q(x_\ell;\bb)-p_{ij}(x_\ell;\ba_{ij})\big|^2\right), & \mbox{if~} 1=\sum_{\ell=1}^mw_\ell|q(x_\ell;\bb)|^2.\end{array}\right.
\end{align}
Therefore, the associated Lagrange dual function at $\bw\ge 0$  is given by 
\begin{equation}\label{eq:Lagrangedualfun2}
\inf_{\begin{subarray}{c} \ba_{ij}\in \bbC^{n_{ij}+1},~\bb\in \bbC^{d+1}\\
            \sum_{\ell=1}^m w_\ell|q(x_\ell;\bb)|^2=1\end{subarray}}\sum_{\ell=1}^m w_\ell \left(\sum_{i=1}^s\sum_{j=1}^t\Big|f_{ij}(x_\ell) q(x_\ell;\bb)-p_{ij}(x_\ell;\ba_{ij})\Big|^2\right).
\end{equation}

To simplify our presentation, associated with \eqref{eq:Lagrangedualfun2}, let 
\begin{equation}\label{eq:notation1}
n = \sum_{i=1}^s\sum_{j=1}^t(n_{ij}+1),~g=st, ~F_{ij}=\diag(f_{ij}(x_1),\dots,f_{ij}(x_m))\in \bbC^{m\times m}, 
\end{equation} 
\begin{align}\label{eq:Ft}
\BF=\left[\begin{array}{c}F_{11} \\\vdots \\F_{s1}\\F_{12} \\\vdots\\F_{st}\end{array}\right]\in \bbC^{gm\times m},~~ {\bf \Theta}=\left[\begin{array}{cccccc}{\bf \Psi}_{11} &&&&&\\&\ddots&&&& \\ &&{\bf \Psi}_{s1}&&&\\ &&&{\bf \Psi}_{21}&& \\&&&&\ddots&\\&&&&&{\bf \Psi}_{st} \end{array}\right]\in \bbC^{gm\times n},
\end{align}
where the row blocks of $\BF$ and the diagonal blocks  of ${\bf \Theta}$ are ordered according to the rule of vectorizing the matrix column-wise (i.e., stacking the $(i+1)$th column below the $i$th).  Similarly, we stack vectors $\ba_{ij}\in\bbC^{n_{ij}+1}$ and $\bp_{ij}\in \bbC^m$ associated with the numerators $p_{ij}$ to have 
\begin{equation}\label{eq:a}
\ba=[\ba_{11}^{\T},\dots,\ba_{s1}^{\T},\ba_{12}^{\T},\dots, \ba_{st}^{\T}]^{\T}\in \bbC^{n},~\bp={\bf \Theta}\ba =[\bp_{11}^{\T},\dots,\bp_{s1}^{\T},\bp_{12}^{\T},\dots, \bp_{st}^{\T}]^{\T}\in \bbC^{mg}
\end{equation}
and define 
\begin{equation}\label{eq:blkW}
{\BW} = I_g\otimes W\in \bbC^{gm\times gm},  ~W=\diag(\bw)\in \bbC^{m\times m},
\end{equation} 
where $\otimes$ stands for the Kronecker product. Consequently, we   rewrite \eqref{eq:Lagrangedualfun2} compactly as 
\begin{equation}\label{eq:Lagrangedualfun3}
\inf_{\begin{subarray}{c} \ba\in \bbC^{n},~\bb\in \bbC^{d+1}\\
             \|\sqrt{W}{\bf \Phi} \bb\|_2=1\end{subarray}}\left\|\sqrt{{\BW}}({\BF} {\bf \Phi} \bb-{\bf \Theta}\ba)\right\|_2^2. 
\end{equation}

An immediate conclusion from the Lagrange duality is  the so-called {\it weak duality}, which asserts that  the dual objective function at any $\bw\ge 0$ is a lower bound for $\eta_\infty$. The following theorem as well as its proof can be viewed as the extension of \cite[Theorem 2.2]{zhyy:2025} to the matrix-valued rational minimax approximation \eqref{eq:bestf0}. 
\begin{theorem}\label{thm:q-dual}
Let $\eta_{\infty}$ be the supremum of \eqref{eq:bestf0}. Then  we have weak duality:
\begin{equation}\label{eq:weak-dual}
    \sup_{\bw\ge \mathbf{0}}d(\bw) =\max_{\bw\in {\cal S}}d(\bw)\le \eta_\infty, 
\end{equation}
where, using the notation in \eqref{eq:notation1}, \eqref{eq:Ft}, \eqref{eq:a} and \eqref{eq:blkW}, the dual function $d(\bw)$ is   
\begin{align}\nonumber
d(\bw):&=\min_{\begin{subarray}{c} \ba_{ij}\in \bbC^{n_{ij}+1},~\bb\in \bbC^{d+1}\\
            \sum_{\ell=1}^m w_\ell|q(x_\ell;\bb)|^2=1\end{subarray}}\sum_{\ell=1}^m w_\ell \left(\sum_{i=1}^s\sum_{j=1}^t\Big|f_{ij}(x_\ell) q(x_\ell;\bb)-p_{ij}(x_\ell;\ba_{ij})\Big|^2\right)\\\label{eq:rat-d-compt}
            &=\min_{\begin{subarray}{c} \ba\in \bbC^{n},~\bb\in \bbC^{d+1}\\
             \|\sqrt{W}{\bf \Phi} \bb\|_2=1\end{subarray}}\left\|\sqrt{{\BW}}({\BF}{\bf\Phi} \bb-{\bf \Theta}\ba)\right\|_2^2,
\end{align}
and ${\cal S}$ is the probability simplex  defined by
\begin{equation}\label{eq:simplex}
{\cal S}:=\{\bw=[w_1,\dots,w_m]^{\T}\in \bbR^m: \bw\ge 0 ~{\rm and } ~\bw^{\T}\be=1\},~~\be=[1,1,\dots,1]^{\T}.
\end{equation}  
\end{theorem}
\begin{proof}
We  first remark that, for any non-negative $\bw\ne \mathbf{0}$, the minimization of \eqref{eq:rat-d-compt} is reachable at a pair $(\ba, \bb)$  as  it is  a trace minimization for a Hermitian positive semi-definite pencil, and  the solution can be obtained by  \cite[Theorems 2.1 and 4.1]{lilb:2013} (more detailed information on computation of $( \ba, \bb)$ will be given in Section \ref{sec_d2}).   Let $(\{\wtd p_{ij}\},\wtd q)$ for  $1\le i\le s,~1\le j\le t$ be the corresponding   polynomials  at the minimizer of \eqref{eq:rat-d-compt}.

We next show that  
$d(\bw)\le \eta_\infty, \forall  \bw\ge \mathbf{0}$, which then  guarantees   
 $\sup_{\bw\ge \mathbf{0}}d(\bw) \le \eta_\infty$. 
 Because $\eta_{\infty}$ is the infimum in \eqref{eq:bestf0}, for any $\epsilon>0$, there is an $R(x)\in  \scrR_{(s,t)}$ with  $q(x_j)\ne 0 ~(\forall x_\ell\in {\cal X})$ so that 
 $$
 \max_{1\le \ell \le m}\|F(x_{\ell})-R(x_{\ell})\|_{\F}^2\le \eta_{\infty}+\epsilon
 $$ 
 which gives 
 \begin{equation}\label{eq:etaeps}
\sum_{i=1}^s \sum_{j=1}^t \left|f_{ij}(x_\ell) -r_{ij}(x_{\ell})\right|^2\le  \eta_\infty+\epsilon, \quad 1\le \ell\le m.
 \end{equation}
As $\bw\ne 0$ and $q(x_\ell)\ne 0 ~(\forall x_\ell\in {\cal X})$, we can choose a scalar $\tau\ne 0$ so that $\what  q = \tau q, ~\what p_{ij}  = \tau p_{ij}$ for $1\le i\le s,~1\le j\le t$ satisfying $\sum_{\ell=1}^m  w_\ell|\what q(x_\ell)|^2 =1$.   Thus, as $(\{\wtd p_{ij}\},\wtd q)$ for  $1\le i\le s,~1\le j\le t$ are the corresponding   polynomials  for \eqref{eq:rat-d-compt} at $\bw$, it holds that
\begin{align*}\nonumber
d(\bw)&=\sum_{\ell=1}^m w_\ell \left( \sum_{i=1}^s\sum_{j=1}^t|f_{ij}(x_\ell) \wtd q(x_\ell)-\wtd   p_{ij}(x_\ell)|^2\right)  \\&\le\sum_{\ell=1}^m w_\ell \left(\sum_{i=1}^s\sum_{j=1}^t|f_{ij}(x_\ell)  \what q (x_\ell)- \what p_{ij} (x_\ell)|^2\right) \quad {\mbox{ (because $(\{\wtd p_{ij}\},\wtd q)$ are optimal)}}\\\label{eq:weakduality}
&= \sum_{\ell=1}^m w_\ell |\what  q(x_\ell)|^2\cdot \left(\sum_{i=1}^s\sum_{j=1}^t\left|f_{ij}(x_{\ell})-\frac{ \what p_{ij}(x_\ell)}{\what q(x_\ell)}\right|^2\right) \\
&=\sum_{\ell=1}^m w_\ell |\what  q(x_\ell)|^2\cdot \left( \sum_{i=1}^s\sum_{j=1}^t\left|f_{ij}(x_{\ell})-r_{ij}(x_\ell)\right|^2\right)\\
& \le  \sum_{\ell=1}^m w_\ell  |\what  q(x_\ell)|^2(\eta_\infty+\epsilon) = \eta_\infty+\epsilon, 
\end{align*}
where  we have used  \eqref{eq:etaeps} and $\sum_{\ell=1}^m w_\ell |\what  q(x_\ell)|^2=1$. Letting $\epsilon\rightarrow 0$ gives the desired $d(\bw)\le \eta_\infty$.

Finally, to prove $\sup_{\bw\ge \mathbf{0}}d(\bw) =\max_{\bw\in {\cal S}}d(\bw),$ we only need to show 
$\sup_{\bw\ge \mathbf{0}}d(\bw) \le \max_{\bw\in {\cal S}}d(\bw)$. For any $\bw\ge \mathbf{0}$, let $\wtd \bw= {\bw}{\varsigma}\in {\cal S}$ with $\varsigma=\frac{1}{\bw^{\T}\be}$ and set correspondingly $\wtd {W}=\diag(\wtd \bw)$ and $\wtd { \BW}:=I_g \otimes\wtd W$. Note 
\begin{align*}
d(\wtd \bw)&=\min_{\begin{subarray}{c}  \ba\in \bbC^{n},~ \bb\in \bbC^{d+1}\\
            \|\sqrt{\wtd {   W}} {\bf \Phi} \bb\|_2 =1\end{subarray}}\left\|\sqrt{\wtd {\BW}}[-{\bf \Theta},\BF {\bf \Phi}]  \left[\begin{array}{c}\ba \\\bb\end{array}\right]\right\|_2^2\\
            &= \min_{\begin{subarray}{c}  \sqrt{\varsigma}\ba\in \bbC^{n},~ \sqrt{\varsigma} \bb\in \bbC^{d+1}\\
            \|\sqrt{W}{\bf \Phi}(\sqrt{\varsigma}\bb)\|_2 =1\end{subarray}}\left\|\sqrt{ \BW}[-{\bf \Theta},\BF {\bf \Phi}]  \left[\begin{array}{c}\sqrt{\varsigma}\ba \\\sqrt{\varsigma} \bb\end{array}\right]\right\|_2^2=d( \bw),
\end{align*}
and thus the proof is complete.
\end{proof}

weak duality suggests a new problem, i.e., the Lagrange dual problem of \eqref{eq:linearity},
\begin{equation}\label{eq:rat-dual}
    \max_{\bw\in {\cal S}}d(\bw) 
\end{equation}
to solve the original matrix-valued rational minimax approximation \eqref{eq:bestf0}. In this framework, the success in computing the minimax solution $R^*\in  \scrR_{(s,t)}$ of \eqref{eq:bestf0},  if it exists, is crucially dependent on the following condition
\begin{equation}\label{eq:strongdual}
d^*:=\max_{\bw\in {\cal S}} d(\bw)=\eta_\infty.
\end{equation} 
The condition \eqref{eq:strongdual} implies that there is no gap between the maximum of the dual problem with $\eta_\infty$, and is known as  {\it strong duality} (see e.g., \cite[Chapter 5.2]{boyd:2004}) in optimization. 

\begin{itemize}\label{rk:dual}

\item[(1)] In the polynomial minimax approximation for a complex-valued function $f$ (i.e.,  $s=t=1,~d=0$), strong duality   is guaranteed (see \cite[Theorem 2.1]{yazz:2023}).  In this case, as $q\in \bbP_0$ and  $\sum_{\ell=1}^m w_\ell |q(x_\ell)|^2=1$,   $|q|\equiv1$; thus  $d(\bw)$ in \eqref{eq:rat-d-compt} becomes
\begin{equation}\nonumber
d(\bw)=\min_{\begin{subarray}{c}  \ba\in \bbC^{n} \end{subarray}}\left\|\sqrt{ \bf W}(q\BF \be-{\bf\Theta} \ba ) \right\|_2^2 \xlongequal{\text{$\wtd \ba = \ba/q$}} \min_{\begin{subarray}{c}  \wtd\ba\in \bbC^{n} \end{subarray}}\left\|\sqrt{W}(\Bf -{\bf\Theta} \wtd \ba )  \right\|_2^2,
\end{equation}   
where ${\Bf}=[f(x_1),\dots, f(x_\ell)]^{\T}$ and ${\bf\Theta}$ is the Vandermonde matrix associated with the nodes $\{x_\ell\}_{\ell=1}^m$. The traditional Lawson's iteration \cite{elwi:1976,laws:1961}, which consists of solving a sequence of weighted least-squares (LS) problems, is shown to be an effective method for solving the dual problem \eqref{eq:rat-dual} (see \cite{yazz:2023}); it  is an   iteratively reweighted least-squares (IRLS) iteration \cite{fint:2018,natr:2020}, and the involved weights   are just the dual variables $\bw$  which are updated according to the errors of the approximations from these weighted LS problems. Interestingly, as a generalization of \cite[Theorem 2.1]{yazz:2023}, for   matrix-valued polynomial minimax approximations with $s\ge 1,t\ge 1$ and $d=0$, Theorem \ref{thm:strongdualPoly} ensures that strong duality \eqref{eq:strongdual} is true.

\item[(2)] For the rational minimax approximation with $s=t=1$ and $d\ge 1$,  \cite[Section 4]{zhyy:2025}  connects strong duality \eqref{eq:strongdual} with Ruttan's sufficient condition \cite[Theorem 2.1]{rutt:1985} (see also \cite[Theorem 3]{this:1993}). In this case, though strong duality cannot be guaranteed for any situation of \eqref{eq:bestf0},  it is true frequently in practice \cite{zhyy:2025}. Theorem \ref{thm:strongdualityRuttan} establishes the counterpart of Ruttan's sufficient condition for   matrix-valued rational minimax approximations.
\end{itemize}


\section{Computation details for the dual function}\label{sec_d2} 
\subsection{Basic properties of the dual function $d(\bw)$}
It is noticed that computing the dual function value $d(\bw)$ at $\bw\in {\cal S}$ needs to solve a constrained minimization \eqref{eq:rat-d-compt}. Fortunately, the dual formulation of the matrix-valued rational minimax approximation \eqref{eq:rat-d-compt} takes a similar form as the one for scalar version (i.e.,  $s=t=1$) given in \cite{zhyy:2025}. Hence, the treatment in \cite[Proposition 3.1]{zhyy:2025} can be employed to transform it into a generalized eigenvalue problem, which is further related to finding the smallest singular pair of a tall matrix. 
\begin{proposition}\label{prop:dual_GEP}
For   $ \bw\in {\cal S}$, with notation in \eqref{eq:notation1}, \eqref{eq:Ft},  \eqref{eq:a} and \eqref{eq:blkW}, we have  
\begin{itemize}
\item[(i)]
$\bc(\bw)=\left[\begin{array}{c} \ba(\bw) \\ \bb(\bw)\end{array}\right] \in \bbC^{n+d+1}$ is a solution of \eqref{eq:rat-d-compt} if and only if it  {is} an eigenvector of the Hermitian positive semi-definite generalized eigenvalue problem $(A_{\bw},B_{\bw})$ and  $d(\bw)$ is the smallest eigenvalue  satisfying
\begin{equation}\label{eq:dual_GEP}
A_{\bw} \bc(\bw)=d(\bw) B_{\bw} \bc(\bw)~ \mbox{ and }~ (\bc(\bw))^{\HH}B_{\bw}\bc(\bw) =1,
\end{equation} 
where
\begin{align}\label{eq:dual_GEPA}
A_{\bw}:&=[-{\bf \Theta}, \BF{\bf \Phi}]^{\HH}\BW[-{\bf \Theta}, \BF{\bf \Phi}]=\left[\begin{array}{cc}{\bf \Theta}^{\HH}\BW{\bf \Theta} & -{\bf \Theta}^{\HH}\BW\BF{\bf \Phi} \\-{\bf \Phi}^{\HH} \BF^{\HH}\BW  {\bf \Theta} & {\bf\Phi}^{\HH}\BF^{\HH}\BW\BF{\bf\Phi}\end{array}\right],\\\label{eq:dual_GEPB}
B_{\bw}:&= [0_{m\times n},{\bf \Phi}]^{\HH} W[0_{m\times n},{\bf\Phi}]=\left[\begin{array}{cc}0_{n\times n} & 0_{n\times (d+1)} \\0_{(d+1)\times n} & {\bf \Phi}^{\HH}W{\bf \Phi} \end{array}\right]\in \bbC^{(n+d+1)\times (n+d+1)};
\end{align}
\item[(ii)] the  Hermitian matrix $H_{\bw}:=A_{\bw} -d(\bw) B_{\bw} \succeq 0$, i.e., $H_{\bw}$ is  positive semi-definite;

\item[(iii)] let $\sqrt{W}{\bf \Phi}=Q_qR_q\in \bbC^{m\times (d+1)}$ and $\sqrt{\BW}{\bf \Theta}=Q_pR_p\in \bbC^{gm\times n}$ be the thin QR factorizations where $Q_q\in \bbC^{m\times \wtd n_2}$, $Q_p\in \bbC^{gm\times \wtd n_1}$, $R_q\in \bbC^{ \wtd n_2 \times (d+1)}$,  $R_p\in \bbC^{\wtd n_1 \times n}$ with $\wtd n_1=\rank(\sqrt{\BW}{\bf \Theta})$ and $\wtd n_2=\rank(\sqrt{\BW}{\bf \Phi})$. Then $d(\bw)$ is the smallest   eigenvalue of the following Hermitian positive semi-definite  matrix 
\begin{equation}\label{eq:Heig}
S(\bw):=S_{\F}-S_{qp}S_{qp}^{\HH}\in \bbC^{\wtd n_2\times \wtd n_2},
\end{equation} 
where 
\begin{equation}\nonumber
S_{\F}=Q_q^{\HH} \left(\sum_{j=1}^t\sum_{i=1}^s|F_{ij}|^2\right)Q_q\in \bbC^{\wtd n_2\times \wtd n_2},~~S_{qp}=Q_q^{\HH}\BF^{\HH} Q_p\in \bbC^{\wtd n_2 \times \wtd n_1}, 
\end{equation}
with $|F_{ij}|=\diag(|f_{ij}(x_1)|,\dots,|f_{ij}(x_m)|)\in \bbC^{m\times m}$. 
Moreover, $\sqrt{d(\bw)}$ is   the smallest singular value of   $(I-Q_pQ_p^{\HH})\BF Q_q\in \bbC^{gm\times  \wtd n_2}$ with the associated singular vector $\what \bb =R_q\bb(\bw)$. Also, $R_p\ba(\bw)=S_{qp}^{\HH}\what \bb$.
\end{itemize}
\end{proposition}
\begin{proof}
With the notation in \eqref{eq:notation1}, \eqref{eq:Ft},  \eqref{eq:a} and \eqref{eq:blkW}, the proof is the same as that for \cite[Proposition 3.1]{zhyy:2025}, and the details are omitted.
\end{proof}


For $\bw \ge 0$ and $W=\diag(\bw)$, we introduce the \(\bw\)-inner product (positive semidefinite): $\langle \by,\bz\rangle_{W}=\by^{\HH} W \bz$ and $\|\by\|_{W}=\sqrt{\by^{\HH}W\by}$, which, for simplicity,  will also be written as  $\langle \by,\bz\rangle_{\bw}$ and $\|\by\|_{\bw}$, respectively. 
Using the \(\bw\)-inner product, we can  rewrite the equations associated with the first \(n\) rows and the last \(d+1\) rows of \(A_{\bw}\bc(\bw)=d(\bw)B_{\bw}\bc(\bw)\) in \eqref{eq:dual_GEP} as  the following optimality conditions: 
 \begin{corollary}\label{coro:verticle_coro1}
  Let \(\bp= {\bf \Theta}\ba(\bw)=[\bp_{11}^{\T},\dots,\bp_{st}^{\T}]^{\T}\in \bbC^{mg}\) and \(\bq={\bf {\bf \Phi}}\bb(\bw)\in \bbC^{m}\) be from the solution of \eqref{eq:Lagrangedualfun3} at \(\bw\in {\cal S}\), where \(\bp_{ij}={\bf \Psi}_{ij}\ba_{ij}\in \bbC^m\) partitioned as in \eqref{eq:a}. Then for any $1\le i\le s$ and $1\le j\le t$, it holds that
    \begin{equation}\label{eq:verticle_coro1}
        F_{ij}\bq-\bp_{ij} \ \bot_{\bw}\  {\rm span}({\bf \Psi}_{ij}),\ \ \BF^{\HH}(\BF \bq-\bp)-d(\bw)\bq\  \bot_{\bw} \ {\rm span}({\bf {\bf \Phi}}),
    \end{equation}
    or equivalently, 
       $ {\BF} \bq-\bp \ \bot_{\BW}\  {\rm span}({\bf \Theta}),\ \ \BF^{\HH}(\BF \bq-\bp)-d(\bw)\bq\  \bot_{W} \ {\rm span}({\bf {\bf \Phi}}).$
\end{corollary}
Furthermore, as a generalization of \cite[Proposition 5.1]{zhha:2025}, we have the explicit form of the gradient of \(d(\bw)\).
\begin{proposition}\label{prop:gradient_d}
  For \(\bw>0\), let \(d(\bw)\) be the smallest eigenvalue of the Hermitian positive semi-definite generalized eigenvalue problem \eqref{eq:Heig}, and \(\bc(\bw)=\begin{bmatrix}
        \ba(\bw)\\ \bb(\bw)
    \end{bmatrix} \in \mathbb{C}^{n+d+1}\) be the associated eigenvector. Denote \(\bp=[\bp_{11}^{\T},\dots,\bp_{st}^{\T}]^{\T}={\bf \Theta}\ba\in \mathbb{C}^{mg}\) and \(\bq={\bf {\bf \Phi}} \bb \in \mathbb{C}^{m}\), where \(\bp_{ij}={\bf \Psi}_{ij} \ba_{ij} \in \mathbb{C}^{m}\) partitioned as in \eqref{eq:a}. If \(d(\bw)\) is a simple eigenvalue, then \(d(\bw)\) is differentiable at \(\bw\) and its gradient is
    \begin{align}\nonumber
        \nabla d(\bw)&=\begin{bmatrix}
            \sum_{i=1}^s\sum_{j=1}^t\left|f_{ij}(x_1)q(x_1)-p_{ij}(x_1)\right|^2-d(\bw)|q(x_{1})|^2\\ \sum_{i=1}^s\sum_{j=1}^t\left|f_{ij}(x_2)q(x_2)-p_{ij}(x_2)\right|^2-d(\bw)|q(x_{2})|^2 \\ \vdots \\ \sum_{i=1}^s\sum_{j=1}^t\left|f_{ij}(x_m)q(x_m)-p_{ij}(x_m)\right|^2-d(\bw)|q(x_{m})|^2
        \end{bmatrix} \\ \label{eq:d_gradient} &=\sum_{i=1}^s\sum_{j=1}^t|F_{ij}\bq-\bp_{ij}|^2-d(\bw)|\bq|^2 \in \mathbb{R}^m.
    \end{align}
\end{proposition}
\begin{proof}
  For \(\bw>0\), we know from Proposition \ref{prop:dual_GEP} that \(\wtd{n}_1=n\) and \(\wtd{n}_2=d+1\), and \(d(\bw)\) is a simple eigenvalue of a positive semi-definite Hermitian matrix $S(\bw)$ in \eqref{eq:Heig} and thus is differentiable at \(\bw\) (see e.g., \cite[Section 7.2.2]{govl:2013}). We obtain the formulation of the gradient \(\nabla d(\bw)\) via computing the directional derivative along a given direction \(\bh\in\mathbb{R}^m\). To this end, let \(\bw(t)=\bw+t\bh \) and \(\bc(\bw(t))\) be the corresponding eigenvector in \eqref{eq:Heig} for any sufficiently small \(t\in \bbR\). We have
  \begin{equation}\nonumber
        \left(A_{\bw(t)}-d(\bw(t))B_{\bw(t)}\right)\bc(\bw(t))=0,\ \ (\bc(\bw(t)))^{\HH}B_{\bw(t)}\bc(\bw(t))=1.
  \end{equation}
  Differentiate the first term with respect to \(t\) and note \(\dot{B}_{\bw(0)}=[0,{\bf \Phi}]^{\HH}{\rm diag}(\bh)[0,{\bf \Phi}]\) and \(\dot{A}_{\bw(0)}=[-{\bf \Theta},\BF{\bf \Phi}]^{\HH}\left(I_g \otimes {\rm diag}(\bh)\right)[-{\bf \Theta},\BF{\bf \Phi}]\) to have
  \begin{equation}\nonumber
        \left(\dot{A}_{\bw(0)}-\dot{d}(\bw(0))B_{\bw(0)}-d(\bw(0))\dot{B}_{\bw(0)}\right)\bc(\bw(0))+(A_{\bw(0)}-d(\bw(0))B_{\bw(0)})\dot{\bc}(\bw(0))=0.
  \end{equation}
  Pre-multiplying \((\bc(\bw(0)))^{\HH}\) on both sides and using 
  \begin{equation}\nonumber
        (\bc(\bw(0)))^{\HH}(A_{\bw(0)}-d(\bw(0))B_{\bw(0)})=0 \ \ {\rm and}\  \ (\bc(\bw(t)))^{\HH}B_{\bw(t)}\bc(\bw(t))=1,
  \end{equation}
  we get
  \begin{align}\nonumber
        (\nabla d(\bw))^{\T}\bh&=\dot{d}(\bw(0))=(\bc(\bw(0)))^{\HH}\dot{A}_{\bw(0)}\bc(\bw(0))-d(\bw(0))(\bc(\bw(0)))^{\HH}\dot{B}_{\bw(0)}\bc(\bw(0)) \\ \nonumber &=\left(\sum_{i=1}^s\sum_{j=1}^t|F_{ij}\bq-\bp_{ij}|^2-d(\bw)|\bq|^2\right)^{\T}\bh,
  \end{align}
leading to the desired result.
\end{proof}
\subsection{Implementation based on Vandermonde with Arnoldi process (V+A)}\label{subsec:VA}
Proposition \ref{prop:dual_GEP}, particularly   item (iii), offers a way to compute the dual function $d(\bw)$. Represented in the monomial basis for $p_{ij}~(1\le i\le s,~1\le j\le t)$ and $q$,  the Vandermonde with Arnoldi process (V+A) presented in \cite{brnt:2021} is effective to handle  the ill-conditionings in the involved Vandermonde basis matrices. V+A has been successfully employed in \cite{hoka:2020,zhsl:2023,zhyy:2025} to resolve the numerical stability in certain   approximation problems.   
 
In our case,  we note that the Vandermonde matrix ${\bf \Phi}=[\mathbf{1}_{m},X\mathbf{1}_{m},\dots,X^{d}\mathbf{1}_{m}]$ (i.e., $\phi_i(x)=x^{i-1},~i\ge 1$) with $X={\rm diag}(\bx)\in \bbC^{m\times m}$ gives 
$$
 \sqrt{W}{\bf \Phi}=[\sqrt{\bw},X\sqrt{\bw},\dots,X^{d}\sqrt{\bw}],
$$
and the orthonormal basis $Q_q$ of ${\rm span}(\sqrt{W}{\bf \Phi})$ can be  obtained via the Arnoldi process  \cite{brnt:2021,hoka:2020,zhsl:2023}. In particular,  assuming $\rank(\sqrt{W}{\bf \Phi})=d+1$ with  $Q_q\be_1=\frac{\sqrt{\bw}}{\|\sqrt{\bw}\|_2}$,
 the Arnoldi process produces (see \cite[Theorem 2.1]{zhsl:2023} and also \cite{hoka:2020})
\begin{equation}\label{eq:ArnoldiQ}
XQ_q=Q_{q}H_q+\gamma_{d+1}  \bq_{d+2}\be_{d+1}^{\T},~\gamma_{d+1} \in \bbC, 
\end{equation}
where $H_q\in \bbC^{(d+1)\times (d+1)}$ is an upper Hessenberg matrix and $\bq_{d+2}\perp Q_q$;  also, the triangular matrix $R_q=\|\sqrt{\bw}\|_2[\be_1,H_q\be_1\dots,H_q^{d}\be_1]$  and $\sqrt{W}{\bf \Phi}=Q_qR_q$ (see \cite[Theorem 2.1]{zhsl:2023}).  
By using the computed $H_d$ and $\gamma_{d+1}$ in \eqref{eq:ArnoldiQ},  the Vandermonde matrix  ${\bf \Phi}(y_1,\dots,y_{\wtd m};d)$ at new nodes $\{y_\ell\}_{j=1}^{\wtd m}$ takes ${\bf \Phi}(y_1,\dots,y_{\wtd m};d)=L_qR_q$ (see \cite[Theorem 2.1]{zhsl:2023}), where $L_q\in \bbC^{\wtd m\times (d+1)}$ satisfies  
\begin{equation}\label{eq:ArnoldiL}
YL_q=L_{q}H_q+\gamma_{d+1}  \bl_{d+2}\be_{d+1}^{\T},~L_q\be_1=\mathbf{1}_{\wtd m},~Y=\diag(y_1,\dots,y_{\wtd m}).
\end{equation}

Related with the numerator polynomial $p_{ij}$ in the rational function $r_{ij}=p_{ij}/q$ and its Vandermonde matrix ${\bf\Psi}_{ij}$ given in \eqref{eq:basisVmatrix}, the same applies to $\sqrt{W}{\bf\Psi}_{ij}$ to yield the orthonormal basis $P_{ij}$ for ${\rm span}(\sqrt{W}{\bf\Psi}_{ij})$. In this procedure, letting $\nu=\max_{ij}n_{ij}$, we only need to compute the orthonormal basis $P_{\nu}$ for ${\rm span}([\sqrt{\bw},X\sqrt{\bw},\dots,X^{\nu}\sqrt{\bw}])$ via  Arnoldi process  \cite{brnt:2021,hoka:2020,zhsl:2023}, and set $P_{ij}=P_{\nu}(:,1:n_{ij})$. Consequently, the matrices $Q_p$ and $R_p$ in the QR factorization of $\sqrt{\BW}{\bf \Theta}$ are 
$$
Q_p= \diag(P_{11},\dots,P_{s1},P_{21},\dots,P_{st}),~~R_p= \diag(R_{11},\dots,R_{s1},R_{21},\dots,R_{st}).
$$
With this and Proposition \ref{prop:dual_GEP},  $\sqrt{d(\bw)}$ is   the smallest singular value of   
\begin{equation}\label{eq:lawsonsvd}
(I-Q_pQ_p^{\HH})\BF Q_q=\left[\begin{array}{c}(I_{n_{11}}-P_{11}P_{11}^{\HH})F_{11}Q_q \\\vdots \\(I_{n_{s1}}-P_{s1}P_{s1}^{\HH})F_{s1}Q_q\\(I_{n_{12}}-P_{12}P_{12}^{\HH})F_{12} Q_q\\\vdots\\(I_{n_{st}}-P_{st}P_{st}^{\HH})F_{st}Q_q\end{array}\right]\in \bbC^{gm\times (d+1)}
\end{equation}
and the associated singular vector $\what \bb$ satisfying  $\what \bb=R_q\bb(\bw)$.  
Let $$\what \ba:=R_p\ba(\bw)=[(R_{11}\ba_{11}(\bw))^{\T},\dots,(R_{s1}\ba_{s1}(\bw))^{\T},(R_{12}\ba_{12}(\bw))^{\T},\dots,(R_{st}\ba_{st}(\bw))^{\T}]^{\T}.$$ 
On the other hand, by  Proposition \ref{prop:dual_GEP},   $\what \ba$ can be computed according to $\what \ba=S_{qp}^{\HH}\what \bb$; denoting  $$\what \ba=[\what \ba_{11}^{\T},\dots,\what \ba_{s1}^{\T},\what \ba_{12}^{\T},\dots,\what \ba_{st}^{\T}]^{\T},~~\what \ba_{ij}\in \bbC^{n_{ij}+1},$$ we then have the correspondence $\what \ba_{ij}=R_{ij}\ba_{ij}(\bw)$. 
 
It is worthy mentioning that computing the vectors $\what \ba$ and $\what \bb$ does not involve explicitly the upper triangular matrices $R_p$ and $R_q$. Moreover, using  $\what \ba$ and $\what \bb$,  we can compute  vectors $\bp_{ij}$ and $\bq$ given in \eqref{eq:vecpij} and \eqref{eq:vecq}, respectively, as  
\begin{subequations}\label{eq:pqvector}
\begin{align}\label{eq:pqvectora}
\bp_{ij} &={\bf \Psi}_{ij}\ba_{ij}(\bw)=W^{-\frac12}P_{ij}(R_{ij}\ba_{ij}(\bw))=W^{-\frac12}P_{ij}\what\ba_{ij},\\
\bq&={\bf \Phi}\bb(\bw) =W^{-\frac12}Q_qR_q\bb(\bw)=W^{-\frac12}Q_q\what\bb,
\label{eq:pqvectorb}
\end{align}
\end{subequations}  
and hence $\br_{ij}=r_{ij}(\bx)= [r_{ij}(x_1),\dots,r_{ij}(x_m)]^{\T}=\bp_{ij}./\bq\in \bbC^m$.

For the evaluation values of $r_{ij}(x)=p_{ij}(x)/q(x)$ at new nodes $\{y_\ell\}_{j=1}^{\wtd m}$, V+A relies on  \eqref{eq:ArnoldiL} and ${\bf \Phi}(y_1,\dots,y_{\wtd m};d)=L_qR_q$  to get 
\begin{equation}\label{eq:evalq}
q(\wtd \by)=[q(y_1),\dots,q(y_{\wtd m})]^{\T}= {\bf \Phi}(y_1,\dots,y_{\wtd m};d)R_q^{-1}(R_q\bb(\bw))=L_q\what \bb;
\end{equation}
the same applies to the numerator polynomial  $p_{ij}(x)$ to get
\begin{equation}\label{eq:evalpij}
p_{ij}(\wtd \by)=[p_{ij}(y_1),\dots,p_{ij}(y_{\wtd m})]^{\T}= {\bf \Psi}_{ij}(y_1,\dots,y_{\wtd m};n_{ij})R_{ij}^{-1}(R_{ij}\ba_{ij}(\bw))=L_{ij}\what \ba_{ij}.
\end{equation}
Thus it leads to the values of $\{r_{ij}(y_{\ell})\}_{\ell=1}^{\wtd m}$ and the vectors $r_{ij}(\by)=p_{ij}(\by)./q(\by)\in \bbC^{\wtd m}$. This technique has   previously been used in, e.g., \cite{hoka:2020,yazz:2023,zhyy:2025}.

\section{Strong duality}\label{sec_strongduality} 
In Section \ref{sbsec:dual}, we have listed some special cases for which strong duality \eqref{eq:strongdual} holds, and thereby, the original minimax approximations can be recovered from solving the dual problem \eqref{eq:rat-dual} in these cases. For the scalar version of rational minimax approximation (i.e., $s=t=1$ and $d\ge 1$), \cite{zhyy:2025} has shown that, strong duality \eqref{eq:strongdual} is equivalent to the classical Ruttan's sufficient condition \cite[Equ. (4.4)]{zhyy:2025} (see also \cite[Theorem 4]{this:1993}). In this section, we shall continue this issue by providing an easily checkable condition for strong duality \eqref{eq:strongdual}. This condition uses the information of the solution of the dual problem \eqref{eq:rat-dual}, and can be viewed as the generalized Ruttan's sufficient condition for the matrix-valued rational minimax approximation  \eqref{eq:bestf0}.

\begin{theorem}\label{thm:strongdualityRuttan}
Let $\bw^*\in{\cal S}$ be the maximizer of the dual  {problem} \eqref{eq:rat-dual}, and $(\ba^*,\bb^*)$ be the associated solution of \eqref{eq:rat-d-compt} that achieves the minimum $d(\bw^*)$, where $\ba^*$ contains the coefficient vectors $\ba_{ij}^*\in \bbC^{n_{ij}+1}$ partitioned as in \eqref{eq:a}. Denote 
$$R^*(x)= \left[\begin{array}{ccc}r_{11}^*(x) &  \cdots & r_{1t}^*(x)   \\\ \vdots& \ddots & \vdots  \\ r_{s1}^*(x) &  \cdots & r_{st}^*(x)\end{array}\right],~~r_{ij}^*(x)=\frac{p_{ij}^*(x)}{q^*(x)}$$ with  
$p_{ij}^*(x)=[\psi_0(x),\dots,\psi_{ij}(x)]\ba_{ij}^*$ and $q(x)=[\phi_0(x),\dots,\phi_d(x)]\bb^*$. Suppose $q^*(x_\ell)\ne 0$ for all $x_\ell\in {\cal X}$. Then  if 
\begin{equation}\label{eq:strongdualityRuttan}
{d(\bw^*)}= \max_{x_{\ell}\in {\cal X}}\|F(x_{\ell})-R^*(x_{\ell})\|_{\F}^2:=e(R^*),
\end{equation}
then strong duality \eqref{eq:strongdual} holds and $R^{*}$ is the global solution to \eqref{eq:bestf0}.
\end{theorem}
\begin{proof}
 We prove it by contradiction. Let $\eta_{\infty}$ be the infimum of \eqref{eq:bestf0}. Suppose $R^{*}$ is not a global solution to \eqref{eq:bestf0}. Then either the infimum of \eqref{eq:bestf0} is not attainable or  is attainable but $R^*$ is not a solution. In either case, it must hold that $e(R^*)>\eta_{\infty}$, and we can find an $\wtd R(x)=\wtd P(x)/\wtd q(x)\in \scrR_{(s,t)}$ with   $\wtd q(x_j)\ne 0~(x_j\in {\cal X})$ satisfying 
 \begin{equation}\label{eq:etawtd}
 e(\wtd R):=\max_{x_{\ell}\in {\cal X}}\|F(x_{\ell})-\wtd R(x_{\ell})\|_{\F}^2<e(R^*).
 \end{equation}
 Indeed, if the infimum of \eqref{eq:bestf0} is unattainable and $e(R^*)>\eta_\infty$, then by the definition of infimum \eqref{eq:bestf0},  we can find an $\wtd  R\in   \scrR_{(s,t)}$  satisfying  $\eta_{\infty}<e(\wtd R)<e(R^*)$; otherwise, if   \eqref{eq:bestf0} is attainable but $R^*$ is not a minimax approximation, then we can take $\wtd R$ as the rational minimax approximant of \eqref{eq:bestf0}. 
 
 Let $\wtd \ba=[\wtd \ba_{11}^{\T},\dots,\wtd \ba_{s1}^{\T},\wtd \ba_{12}^{\T},\dots,\wtd \ba_{st}^{\T}]^{\T}$ and $\wtd \bb$ be the associated coefficients vectors of $\{\wtd p_{ij}(x)\}$ and $\wtd q(x)$ in the basis of $\{\psi_{j}\}$ and $\{\phi_{j}\}$. By item (ii) of Proposition \ref{prop:dual_GEP}, the matrix $H_{\bw^*}=A_{\bw^*}-d(\bw^*)B_{\bw^*}$ is positive semi-definite where $A_{\bw^*}$ and $B_{\bw^*}$ are defined in \eqref{eq:dual_GEPA} and \eqref{eq:dual_GEPB}, respectively. This gives 
 \begin{align}\nonumber
 0\le& \left[\begin{array}{c} \wtd \ba  \\ \wtd\bb \end{array}\right]^{\HH}\left(A_{\bw^*}-d(\bw^*)B_{\bw^*}\right) \left[\begin{array}{c} \wtd\ba \\ \wtd\bb \end{array}\right] \\\nonumber
 =&\left[\begin{array}{c} \wtd \ba  \\ \wtd\bb \end{array}\right]^{\HH}[-{\bf \Theta}, \BF{\bf \Phi}]^{\HH}\BW^*[-{\bf \Theta}, \BF{\bf \Phi}]\left[\begin{array}{c} \wtd \ba  \\ \wtd\bb \end{array}\right] -d(\bw^*)\left[\begin{array}{c} \wtd \ba  \\ \wtd\bb \end{array}\right]^{\HH}[0_{m\times n},{\bf \Phi}]^{\HH} W^*[0_{m\times n},{\bf\Phi}]\left[\begin{array}{c} \wtd \ba  \\ \wtd\bb \end{array}\right]\\\nonumber
 =&\sum_{\ell=1}^m w_\ell^* \left(\sum_{i=1}^s\sum_{j=1}^t|f_{ij}(x_\ell)\wtd q(x_\ell)-\wtd  p_{ij}(x_\ell)|^2 -d(\bw^*)|\wtd  q(x_\ell)|^2\right)\\\nonumber
=&\sum_{\ell=1}^mw_\ell^*| \wtd  q(x_\ell)|^2\left(\sum_{i=1}^s\sum_{j=1}^t|f_{ij}(x_\ell)-\wtd r_{ij}(x_\ell)|^2-e(R^*)\right)\quad \quad {\rm (by ~\eqref{eq:strongdualityRuttan})}\\
=&\sum_{\ell=1}^mw_\ell^*| \wtd  q(x_\ell)|^2\left(\|F(x_\ell)-\wtd R(x_\ell)\|_{\F}^2-e(R^*)\right). 
 \end{align}
 This implies that there  exists at least one $x_\ell$ so that $\|F(x_\ell)-\wtd R(x_\ell)\|_{\F}^2\ge e(R^*)$, and therefore, 
 $$
e(\wtd R)=\max_{x_\ell\in {\cal X}} \|F(x_\ell)-\wtd R(x_\ell)\|_{\F}^2\ge e(R^*).
 $$
But this is a contradiction due to \eqref{eq:etawtd}. This proves that the infimum in  \eqref{eq:bestf0} is attainable by $R^*\in \scrR_{(s,t)}.$ Thereby, we have $\eta_{\infty}=e(R^*)$. Since $\eta_{\infty}$ is an upper  bound of the dual problem by weak duality \eqref{eq:weak-dual},  $\eta_{\infty}=e(R^*)$ ensures strong duality \eqref{eq:strongdual}. The proof is complete.
\end{proof}
 
 We remark that  \eqref{eq:etawtd} is a sufficient condition to guarantee the solvability of  \eqref{eq:bestf0} and strong duality \eqref{eq:strongdual}. The following theorem shows that for the matrix-valued polynomial minimax approximation with $s\ge 1,t\ge 1$ and $d=0$ in \eqref{eq:bestf0}, strong duality \eqref{eq:strongdual} holds.  
 \begin{theorem}\label{thm:strongdualPoly}
 For the matrix-valued polynomial minimax approximation with $s\ge 1,t\ge 1$ and $d=0$ in \eqref{eq:rats}, strong duality \eqref{eq:strongdual} is true and the matrix-valued minimax polynomial $P^*$ of \eqref{eq:bestf0} can be solved from the dual problem \eqref{eq:rat-dual}.  
\end{theorem} 
 \begin{proof}
 For the polynomial case, we first note that \eqref{eq:bestf0} is equivalent to 
  \begin{align}\nonumber
&\inf_{\eta\in \bbR,~p_{ij}\in \bbP_{n_{ij}}}\eta \\\label{eq:linearitypoly}
 s.t., ~&\|F(x_\ell)-P(x_\ell)\|_{\F}^2\le \eta, \quad 1\le \ell\le m,
\end{align}
where $P:\bbC\rightarrow \bbC^{s\times t}$ is the matrix-valued polynomial with $p_{ij}\in \bbP_{n_{ij}}$ as  its $(i,j)$th component. 
Using notation in \eqref{eq:Ft}, we know that   each constraint in \eqref{eq:linearitypoly} can be written as  
$$
\|\BF\be_{\ell}-\Pi_\ell{\bf \Theta} \ba\|_2^2 - \eta\le 0,
$$
which is convex with respect to $\ba$ and $\eta$, where 
$$\Pi_\ell=\diag(\underbrace{\be_\ell \be_\ell^{\T},\dots, \be_\ell \be_\ell^{\T}}_{g})\in \bbR^{mg\times mg},~\be_\ell\in \bbR^m.$$  Therefore, \eqref{eq:linearitypoly} is a convex minimization and the infimum is achievable. Note that \eqref{eq:rat-dual} is the dual problem of \eqref{eq:linearitypoly}. Also, the Slater condition (see e.g., \cite[Section 5.2.3]{boyd:2004}) holds for \eqref{eq:linearitypoly} since there is $(\eta,  \ba)$ so that $\|\BF\be_{\ell}-\Pi_\ell{\bf \Theta} \ba\|_2^2 - \eta< 0$ for any $\ell=1,2,\dots,m$. Hence, strong duality \eqref{eq:strongdual} is ensured in this case by \cite[Section 5.2.3]{boyd:2004}.
 \end{proof}
In the framework of the dual problem \eqref{eq:rat-dual}, we point out that the condition \eqref{eq:strongdualityRuttan} for strong duality \eqref{eq:strongdual} is numerically checkable. In particular, when an   approximation $\bw\in {\cal S}$ of the global maximizer of the dual \eqref{eq:rat-dual}  together with the associated $R(x)$ is available, we then can compute the relative error of strong duality 
\begin{equation}\label{eq:reldual}
\epsilon(\bw):=\left|\frac{ e(R)-{d(\bw)}}{e(R)}\right|,~~{\rm where}~~e(R)=\max_{x_{\ell}\in {\cal X}}\|F(x_{\ell})-R(x_{\ell})\|_{\F}^2,
\end{equation}
as a measure of the accuracy of the rational approximation $R(x)$.

\section{Extreme points and complementary slackness}\label{sec:complement}
A noticeable property, known as the equioscillation property \cite[Theorem 24.1]{tref:2019a},  of the real polynomial minimax approximation describes the alternation of the maximum error occurring in the so-called  {\it extreme points} ({aka}  {\it reference points}). In our case, particularly, for the solution $R^*:\bbC\rightarrow \bbC^{s\times t}$ of \eqref{eq:bestf0},  extreme points refer to nodes  that  {achieve} $\|F(x_\ell)- R^*(x_\ell)\|_{\F}^2=e(R^*):=\max_{x_{j}\in {\cal X}}\|F(x_{j})-R^*(x_{j})\|_{\F}^2$. Those points are denoted by
\begin{equation}\label{eq:extremalset}
 {\cal  X}_e(R^*):=\left\{x_\ell\in {\cal X}:\|F(x_\ell)- R^*(x_\ell)\|_{\F}^2=e(R^*)\right\}. 
\end{equation}
In the following theorem,  under  strong duality \eqref{eq:strongdual},  we prove a complementary slackness and also show that  non-extreme points in ${\cal X}$ can be filtered out to reduce the computational costs and accelerate the convergence for solving the dual  {problem} \eqref{eq:rat-dual}. This is a generalization of \cite[Theorem 2.4]{zhyy:2025} to the matrix-valued functions. The same strategy has been used in Lawson's iteration for the linear  \cite{clin:1972,laws:1961,rice:1969,yazz:2023} and rational \cite{zhyy:2025} minimax approximations.

\begin{theorem}\label{thm:complement}
Suppose $R^*(x)=[r_{ij}^*(x)]\in \scrR_{(s,t)}$ is a solution to \eqref{eq:bestf0} with   $q^*(x_j)\ne 0~(\forall x_j\in {\cal X})$. If  strong duality \eqref{eq:strongdual} holds, then
\begin{itemize}
\item[(i)] for any solution $\bw^*$ of the   dual {problem} \eqref{eq:rat-dual} with  
  $d(\bw^*)$  given in \eqref{eq:rat-d-compt}, we have 
\begin{equation}\label{eq:complement}
 {\rm (complementary ~slackness)}\quad    w_\ell^*(e(R^*)-\|F(x_\ell)-R^*(x_\ell)\|_{\F}^2)=0, \quad\forall 1\le\ell\le m;
\end{equation}
\item[(ii)] {\rm (number of extreme points)} there are at least $\min_{1\le i\le s, 1\le j\le t}(n_{ij}+2)$ extreme points associated with $R^*(x)$, i.e., 
\begin{equation}\label{eq:NoExtr}
|{\cal  X}_e(R^*)|\ge \min_{1\le i\le s, 1\le j\le t}(n_{ij}+2);
\end{equation}
\item[(iii)]  {\rm (filtering out non-extreme points)}  for any subset $\breve{\cal X} \subseteq{\cal X}$ satisfying ${\cal X}_e(R^*)\subseteq \breve{\cal X}$,   $(\eta_\infty, \{p_{ij}^*\},q^*)$ solves 
\begin{align}\nonumber
&\inf_{\eta\in \bbR,~p_{ij}\in \bbP_{n_{ij}},~0\not \equiv q\in \bbP_{d}}\eta \\\label{eq:linearitysubset}
 s.t., ~&\sum_{i=1}^s \sum_{j=1}^t \left|f_{ij}(x_\ell)q(x_{\ell})-p_{ij}(x_{\ell})\right|^2\le \eta |q(x_{\ell})|^2, \quad x_\ell \in \breve{\cal X} 
\end{align}
for which   strong duality  holds too; that is, 
\begin{equation}\label{eq:dualPsubX}
\eta_\infty =\max_{\breve{\bw}\in \breve{{\cal S}}} \breve{d}(\breve{\bw}),
\end{equation}  
where $\breve{{\cal S}}=\{\breve{\bw}=[\breve{w}_1,\dots,\breve{w}_{\breve{s}}]^{\T}\in \bbR^{\breve{s}}: \breve{\bw}\ge 0 ~{\rm and } ~\breve{\bw}^{\T}\be=1\},
$ $\breve{s}=|\breve{{\cal X}}|$, and
\begin{equation}\label{eq:rat-d-subset}
\breve{d}(\breve{\bw})=\min_{\begin{subarray}{c} p_{ij}\in \bbP_{n_{ij}},~ q\in \bbP_{d}\\
            \sum_{x_\ell\in\breve{\cal X}} \breve{w}_\ell |q(x_\ell)|^2=1\end{subarray}}~\sum_{x_\ell\in\breve{\cal X}} \breve{w}_\ell \sum_{i=1}^s \sum_{j=1}^t |f_{ij}(x_\ell) q(x_\ell)-p_{ij}(x_\ell)|^2.
\end{equation}  
Moreover, for any solution $\breve{\bw}^*$ of the dual \eqref{eq:dualPsubX},  $(\{p_{ij}^*\},q^*)$ achieves the minimum  $\breve{d}(\breve{\bw}^*)=\eta_\infty$ of \eqref{eq:rat-d-subset}.
\end{itemize}
\end{theorem}
\begin{proof}
The proofs for items (i) and (iii)  follow that of \cite[Theorem 2.4]{zhyy:2025} with suitable modifications according to the block structure of \eqref{eq:bestf0}. 

For (i), as strong duality \eqref{eq:strongdual} holds and   $(\{p_{ij}^*\},q^*)$ are feasible (by a scaling for $q^*$ using $\bw^*$ to ensure $\sum_{\ell=1}^m w^*_\ell|q^*(x_\ell)|^2=1$) for the minimization \eqref{eq:rat-d-compt} at $\bw=\bw^*$, it holds  
\begin{align*}
\eta_\infty=d(\bw^*)& \le \sum_{\ell=1}^m w^*_\ell \left(\sum_{i=1}^s\sum_{j=1}^t|f_{ij}(x_\ell)q^*(x_\ell)-p_{ij}^*(x_\ell)|^2\right) \\
&=  \sum_{\ell=1}^m w^*_\ell|q^*(x_\ell)|^2\left(\sum_{i=1}^s\sum_{j=1}^t |f_{ij}(x_\ell)-R^*(x_\ell)|^2\right)\\
&=  \sum_{\ell=1}^m w^*_\ell|q^*(x_\ell)|^2\|F(x_\ell)-R^*(x_\ell)\|_{\F}^2\\
&\le  \eta_\infty \sum_{\ell=1}^m w^*_\ell|q^*(x_\ell)|^2=\eta_\infty;
\end{align*}
this gives $$w^*_\ell|q^*(x_\ell)|^2 \left(\eta_\infty-\|F(x_\ell)-R^*(x_\ell)\|_{\F}^2\right)=0, ~\forall 1\le  \ell\le m.$$ Because $q^*(x_\ell)\ne 0,~\forall x_\ell\in {\cal X}$, we have \eqref{eq:complement}.

For item (ii), the case $e(R^*)=0$ is trivial as   ${\cal  X}_e(R^*)={\cal X}$. For the case $e(R^*)>0$, by \eqref{eq:complement}, we first have $|{\cal J}_{\bw^*}|\le |{\cal  X}_e(R^*)|$ where
\begin{equation}\label{eq:Jw}
{\cal J}_{\bw^*}:=\{\ell: w_\ell^*\ne 0,~1\le \ell\le m\}.
\end{equation}
Now, assume $|{\cal  X}_e(R^*)|< \min_{1\le i\le s, 1\le j\le t}(n_{ij}+2)$, and assume w.l.g. that ${\cal J}_{\bw^*}=\{1,\dots,k\}$ with $1\le k\le \min_{1\le i\le s, 1\le j\le t}(n_{ij}+1)$. The later implies that for each $(i,j)$, we can find an interpolation polynomial $\tilde p_{ij}(x)\in \bbP_{n_{ij}}$ such that
\begin{equation}\label{eq:interpcond}
f_{ij}(x_\ell)q^*(x_\ell)=\tilde p_{ij}(x_\ell),~~\ell=1,2,\dots,k.
\end{equation}
Furthermore, by strong duality,
\begin{align*}
\eta_\infty&=e(R^*)=d(\bw^*)\\
&\le \sum_{\ell=1}^m w^*_\ell \left( \sum_{i=1}^s\sum_{j=1}^t|f_{ij}(x_\ell)   q^*(x_\ell)-\tilde   p_{ij}(x_\ell)|^2\right) ~~\quad\mbox{(as $(\{\tilde p_{ij}\}, q^*)$) may not be optimal)}\\
&=\sum_{\ell=1}^k w^*_\ell \left( \sum_{i=1}^s\sum_{j=1}^t|f_{ij}(x_\ell)   q^*(x_\ell)-\tilde   p_{ij}(x_\ell)|^2\right) ~~\quad(w^*_\ell=0,~k<\ell\le m)\\
&=0
\end{align*}
where the last equality is due to the interpolation conditions \eqref{eq:interpcond}. This contradicts against $e(R^*)>0$, and hence the assertion (ii) follows.

To prove (iii), according to  \eqref{eq:complement}, we know for any solution $\bw^*=[w_1^*,\dots,w_m^*]^{\T}$ of the dual problem \eqref{eq:rat-dual}, $w_j^*=0$ for any $j$ with $x_j\not\in {\cal X}_e(R^*)$ in \eqref{eq:extremalset}. For simplicity, assume $ {\cal X}_e(R^*)=\{x_j\}_{j=1}^{s}$ and $\breve{\cal X}=\{x_j\}_{j=1}^{\breve{s}}$ with $\breve{s}\ge s$. Denote by $\breve{\eta}_\infty$ the infimum of \eqref{eq:linearitysubset} and   $\breve{d}^*=\max_{\breve{\bw}\in \breve{{\cal S}}} \breve{d}(\breve{\bw})$. Similarly to weak duality \eqref{eq:weak-dual} in Theorem \ref{thm:q-dual}, we have   weak duality $\breve{d}^*\le \breve{\eta}_\infty$; note that this weak duality holds without requiring the attainability condition of \eqref{eq:linearitysubset}.  Also, as $(\eta_\infty, \{p_{ij}\}^*,q^*)$ are feasible for \eqref{eq:linearitysubset}, $\breve{\eta}_\infty\le {\eta}_\infty$; thus 
\begin{equation}\label{eq:subdual1}
\breve{d}^*=\breve{d}(\breve{\bw}^*)\le \breve{\eta}_\infty\le {\eta}_\infty=d(\bw^*)
\end{equation} 
where $\breve{\bw}^*$ and  $ {\bw}^*$ are solutions  for  the dual problems \eqref{eq:dualPsubX} and  \eqref{eq:rat-dual}, respectively.  On the other hand, note, by the complementary slackness, that $w_j^*=0$ for any $j\ge s$, and hence
\begin{align*}
{d}({\bw}^*)&=\min_{\begin{subarray}{c}p_{ij}\in \bbP_{n_{ij}},~ q\in \bbP_{d}\\
            \sum_{\ell=1}^m  {w}_\ell^* |q(x_\ell)|^2=1\end{subarray}}\sum_{\ell=1}^m  {w}_\ell^* \left(\sum_{i=1}^s\sum_{j=1}^t |f_{ij}(x_\ell) q(x_\ell)-p_{ij}(x_\ell)|^2\right)\\
            &=\min_{\begin{subarray}{c} p_{ij}\in \bbP_{n_{ij}},~q\in \bbP_{d}\\
            \sum_{\ell=1}^{\breve{s}}  {w}_\ell^* |q(x_\ell)|^2=1\end{subarray}}\sum_{\ell=1}^{\breve{s}}  {w}_\ell^* \left(\sum_{i=1}^s\sum_{j=1}^t|f_{ij}(x_\ell) q(x_\ell)-p_{ij}(x_\ell)|^2\right) = \breve{d}({\bw}^*_{[1:\breve{s}]}) \le \breve{d}^*,
\end{align*}
where the last inequality follows because $\breve{d}^*$ is the maximum of \eqref{eq:dualPsubX} over all $\breve\bw\in \breve{\cal S}$. Together with \eqref{eq:subdual1}, we conclude   $\breve{d}^*=\breve{d}(\breve{\bw}^*)= \breve{\eta}_\infty= {\eta}_\infty=d(\bw^*)$. Finally, for any solution $\breve{\bw}^*$ of the dual \eqref{eq:dualPsubX}, suppose $(\{\hat p_{ij}\},\hat q)$ achieve the minimum  $\breve{d}(\breve{\bw}^*)=\eta_\infty$ of \eqref{eq:rat-d-subset}. Scale $(\{ p_{ij}^*\},  q^*)$ to satisfy the constraint $\sum_{\ell=1}^{\breve{s}}  \breve{w}_\ell^* |q^*(x_\ell)|^2=1$ (i.e., $(\{ p_{ij}^*\}, q^*)$ are feasible for the minimization \eqref{eq:rat-d-subset} with $\breve{\bw}=\breve{\bw}^*$);  thus noticing  
\begin{align*}
\breve{d}^*=\breve{d}(\breve{\bw}^*)& = \sum_{\ell=1}^{\breve{s}} \breve{w}^*_\ell\left(\sum_{i=1}^s\sum_{j=1}^t|f_{ij}(x_\ell)\hat q(x_\ell)-\hat p_{ij}(x_\ell)|^2\right)\\
&\le  \sum_{\ell=1}^{\breve{s}} \breve{w}^*_\ell\left(\sum_{i=1}^s\sum_{j=1}^t|f_{ij}(x_\ell) q^*(x_\ell)- p_{ij}^*(x_\ell)|^2\right)\\
&=  \sum_{\ell=1}^{\breve{s}} \breve{w}^*_\ell |q^*(x_\ell)|^2 \|F(x_\ell)-R^*(x_\ell)\|_{\F}^2 \le  \eta_\infty\sum_{\ell=1}^{\breve{s}} \breve{w}^*_\ell|q^*(x_\ell)|^2=\eta_\infty=\breve{d}^*,
\end{align*}
we assert that $(\{p_{ij}^*\},q^*)$ achieves the minimum  $\breve{d}(\breve{\bw}^*)=\eta_\infty$ of \eqref{eq:rat-d-subset}.
\end{proof} 
As a consequence of Theorem \ref{thm:strongdualPoly} and item (ii) of Theorem \ref{thm:complement},  we have 
\begin{corollary}\label{cor:extPntpoly}
For the matrix-valued polynomial minimax approximation with $s\ge 1,t\ge 1$ and $d=0$ in \eqref{eq:rats}, $|{\cal  X}_e(P^*)|\ge \min_{1\le i\le s, 1\le j\le t}(n_{ij}+2)$   for the solution  $P^*$ of \eqref{eq:bestf0}.
\end{corollary}

\section{{\sf m-d-Lawson}: a method for rational minimax approximation of matrix-valued functions}\label{sec_lawson} 
Lawson's iteration \cite{laws:1961} is a traditional method for solving the discrete linear minimax approximations of complex-valued functions. Besides some previous convergence of Lawson's iteration for the polynomial minimax approximation (see e.g., \cite{bama:1970,clin:1972,elwi:1976,rice:1969}), recent works in \cite{yazz:2023,zhha:2025,zhyy:2025} have revealed that Lawson's iteration is indeed a method for solving the associated dual problems. Closely relying upon this connection,  \cite{zhha:2025,zhyy:2025} proposed a dual Lawson's iteration, namely, {\tt d-Lawson},  for the discrete rational minimax approximation, and established basic convergence analysis. Along this way, we propose a dual Lawson's iteration (Algorithm \ref{alg:mdLawson})  for the matrix-valued rational minimax approximation \eqref{eq:bestf0}. In our paper, this algorithm is termed    {\sf m-d-Lawson}, indicating its {\sf d}ual-based {\sf Lawson} iteration for {\sf m}atrix-valued functions.

\begin{algorithm}[h!!!]
\caption{A matrix-valued dual Lawson's iteration ({\sf m-d-Lawson}) for \eqref{eq:bestf0}} \label{alg:mdLawson}
\begin{algorithmic}[1]
\renewcommand{\algorithmicrequire}{\textbf{Input:}}
\renewcommand{\algorithmicensure}{\textbf{Output:}}
\REQUIRE Given   $\{(x_\ell,F(x_\ell))\}_{\ell=1}^m$  with $x_\ell\in \Omega$, a   tolerance for strong duality $\epsilon_r>0$, the maximum number $k_{\rm maxit}$ of iterations, a new vector $\by=[y_1,\dots,y_{\wtd m}]^{\T}\in \bbC^{\wtd m}$ of nodes; 
\ENSURE   the evaluations   $R^*(y_\ell)\in \bbC^{s\times t}$ of the rational minimax approximant  $R^*\in  \scrR_{(s,t)}$ of \eqref{eq:bestf0} at  $y_\ell$ for $1\le \ell\le \wtd m$.
        \smallskip

\STATE  (Initialization) Let $k=0$; choose $0<\bw^{(0)}\in {\cal S}$  and a  tolerance $\epsilon_w$ for the weights;
\STATE (Filtering) Remove nodes $x_\ell$ with $w_\ell^{(k)}<\epsilon_w$;
\STATE Compute  $d(\bw^{(k)})$ and the associated  $r_{ij}^{(k)}(\bx)=\bp_{ij}^{(k)}./\bq^{(k)}$ according to \eqref{eq:pqvector};
\STATE (Stop rule and evaluation)  Stop either if $k\ge k_{\rm maxit}$ or 
\begin{equation}\label{eq:stop}
\epsilon(\bw^{(k)}):=\left|\frac{ {e(R^{{(k)}})-d(\bw^{(k)})}}{e(R^{{(k)}})}\right|<\epsilon_r,~~{\rm where}~~e(R^{{(k)}})=\max_{x_\ell\in {\cal X}}\|F(x_\ell)-R^{(k)}(x_\ell)\|_{\F}^2,
\end{equation}
and compute $r_{ij}^*(\by)=p_{ij}^*(\by)./q^*(\by)\in \bbC^{\wtd m}$, $\{R^*(y_\ell)\}_{\ell=1}^{\wtd m}$ according to \eqref{eq:evalpij} and \eqref{eq:evalq}; 

\STATE (Updating weights) Update $\bw^{(k+1)}$ according to   ($\beta > 0$)
\begin{equation}\nonumber
w_\ell^{(k+1)}=\frac{w_\ell^{(k)}\|F(x_\ell)-R^{(k)}(x_\ell)\|_{\F}^{\beta}}{\sum_{i}w_i^{(k)}\|F(x_i)-R^{(k)}(x_i)\|_{\F}^{\beta}},~~\forall \ell,
\end{equation} 
and   go to Step 2 with $k=k+1$.
\end{algorithmic}
\end{algorithm}  

In each iteration,  {\sf m-d-Lawson} consists of a calling of a thin SVD in Step 3 (see \eqref{eq:lawsonsvd}) and  a basic updating of the weight vector  $\bw$ in Step 5. The parameter $\beta>0$ involving in Step 5 is called the {\it Lawson exponent} \cite{clin:1972,laws:1961} and is a factor for convergence.  In our numerical testing, we choose $\beta=1$, and set  $\epsilon_w=0, \epsilon_r=10^{-3}$ in Step 2 and Step 4, respectively.

Regarding the convergence of {\sf m-d-Lawson}, we   prove in Section \ref{sec:convergence} that  for polynomial minimax approximations, the choice $\beta=1$ is near-optimal, and the sequence $\{d(\bw^{(k)})\}$ monotonically increases for any $\beta\in (0,2)$;  for   rational minimax approximations,  any sufficiently small $\beta>0$   generically ensures the monotonic convergence. These results generalize the scalar-case ($s=t=1$) convergence properties established in \cite{zhha:2025}.

\section{Convergence analysis of {\sf m-d-Lawson}}\label{sec:convergence}

Throughout this section, we set $\epsilon_w=0$ in Step 2 of Algorithm \ref{alg:mdLawson} and also assume  the following two fundamental conditions:
\begin{itemize}
    \item [(\textbf{A1})] \(q_{\ell}^{(k)}:=\be_\ell^{\T}\bq^{(k)}=q^{(k)}(x_{\ell}) \ne 0\) for any \(1\le \ell\le m\) and any \(k \ge 0\);
    \item [(\textbf{A2})] the cardinality \(\left|{\cal J}_{\bw^{(k)}}\right| \ge \max\{d+1,\max_{1 \le i \le s,1 \le j \le t}\{n_{ij}+1\}\}\), where \({\cal J}_{\bw^{(k)}}:=\{\ell: w_{\ell}^{(k)} \ne 0,~1\le \ell\le m\}\).
\end{itemize}
\begin{proposition}\label{prop:index_update}
    If \(d(\bw^{(k)}) > 0\), then
    \begin{equation}\label{eq:index_update}
        {\cal J}_{\bw^{(k+1)}}=\{\ell: w_{\ell}^{(k)}\tau_{\ell}^{(k)} \ne 0, ~1\le \ell\le m\},\ \ {\rm where}\ \ \tau_{\ell}^{(k)}:=\|F(x_{\ell})-R^{(k)}(x_{\ell})\|_{\F}.
    \end{equation}
\end{proposition}
\begin{proof}
    By   (\textbf{A1}),  \(d(\bw^{(k)})>0\) implies \(\sum\limits_{\ell=1}^{m}w^{(k)}_{\ell}\|F(x_j)-R^{(k)}(x_\ell)\|_{\F}^{\beta}>0\). According to Lawson's updating   in Step 5 of Algorithm \ref{alg:mdLawson}, \eqref{prop:index_update} holds as \(w_{\ell}^{(k+1)} = 0 \Longleftrightarrow w_{\ell}^{(k)}\tau_{\ell}^{(k)} = 0\).
\end{proof}

\subsection{A lower bound of the dual objective function value}\label{sbsec:lowerbound_d}
From the perspective of dual problem \eqref{eq:rat-dual} optimization, Lawson's iteration can be interpreted as a gradient ascent method. Indeed, suppose the current \(0 < \bw \in {\cal S}\), \(q_{j}=\bq^{\rm T}\be_{j} \ne 0, \ \forall 1\le j\le m\) and \(d(\bw)\) is the simple eigenvalue of the matrix pencil \((A_{\bw},B_{\bw})\), then 
\begin{equation}\nonumber
    \bg(\bw):={\rm diag}\left(\frac{w_1}{|q_1|^2},\dots,\frac{w_m}{|q_m|^2}\right)\nabla d(\bw) \in \mathbb{R}^m
\end{equation}
is an ascent direction for the dual function \(d(\bw)\) because by Proposition \ref{prop:gradient_d}
\begin{equation}\nonumber
    \bg(\bw)^{\rm T}\nabla d(\bw)=\nabla d(\bw)^{\rm T}{\rm diag}\left(\frac{w_1}{|q_1|^2},\dots,\frac{w_m}{|q_m|^2}\right)\nabla d(\bw)>0.
\end{equation}
Relying on this direction and the step-size \(s=\frac{1}{d(\bw)}>0\), the   update 
\begin{equation}\nonumber
 \bw+s\bg=\frac{1}{d(\bw)}\begin{bmatrix}
        w_1\|F(x_1)-R(x_1)\|_{\F}^2 \\ \vdots \\ w_m\|F(x_m)-R(x_m)\|_{\F}^2
    \end{bmatrix}
\end{equation}
can be  scaled to have \(\wtd \bw= \frac{\bw+s\bg}{\be^{\T}(\bw+s\bg)} \in {\cal S}\). This says that  an iteration of {\sf m-d-Lawson}  in Algorithm \ref{alg:mdLawson} with the Lawson exponent \(\beta=2\) is an ascent projection gradient iteration. While this connection provides valuable insight into {\sf m-d-Lawson}, proving the precise monotonic convergence of the sequence $\{d(\bw^{(k)})\}_k$ requires further investigation. This constitutes the main goal of this section.
\begin{lemma}\label{lem:lem_verticle1}
Let $(\ba^{(k)},\bb^{(k)})$ be the solution to \eqref{eq:Lagrangedualfun3} at $\bw^{(k)}\in {\cal S}$ and  $(\{p_{ij}^{(k)}\},q^{(k)})$ be the associated polynomials.  Assume (\textbf{A1}) and (\textbf{A2}). 
    Let \(\bp^{(k)}={\bf \Theta} \ba^{(k)}$, $\bq^{(k)}={\bf \Phi} \bb^{(k)}$, $\mathbf{1}_g=[1,\dots,1]^{\rm T} \in \mathbb{R}^{g}$  with  $g=st$,  and
    \begin{align*}
\boldsymbol{\tau}^{(k)}(\beta):= [(\tau_1^{(k)})^\beta,\dots, (\tau_m^{(k)})^\beta]^{\T}\in \bbR^m, 
    \end{align*}
    where $\tau_{\ell}^{(k)}=\|F(x_{\ell})-R^{(k)}(x_{\ell})\|_{\F}~(1\le \ell \le m)$ and \(R^{(k)}(x)=[p^{(k)}_{ij}(x)/q^{(k)}(x)]:\mathbb{C} \rightarrow \mathbb{C}^{s \times t}\). If \(d(\bw^{(k)})>0\), then
    \begin{equation}\label{eq:yperp}
        \frac{\left(\BF\bq^{(k)}-\bp^{(k)}\right).}{\mathbf{1}_g \otimes \boldsymbol{\tau}^{(k)}(\beta)}\  \bot_{{\BW}^{(k+1)}}\ {\rm span}({\bf \Theta}),
    \end{equation}
where ${\BW}^{(k+1)}$ is defined by \eqref{eq:blkW} and we set \(0\frac{0}{0}=0\) for convenience.
\end{lemma}
\begin{proof}
Denote \(\by=\frac{\left(\BF\bq^{(k)}-\bp^{(k)}\right).}{\mathbf{1}_g \otimes \boldsymbol{\tau}^{(k)}(\beta)}\) and 
 \begin{equation}\label{eq:gammak}
 \gamma_{\beta}^{(k)}:=\sum\limits_{\ell=1}^{m}w_{\ell}^{(k)}\|F(x_{\ell})-R^{(k)}(x_{\ell})\|_{\F}^{\beta}=\sum\limits_{\ell=1}^{m}w_{\ell}^{(k)}(\tau_\ell^{(k)})^{\beta}.
 \end{equation}
 To simplify notation, we drop the iteration superscript (e.g. $\bp \equiv \bp^{(k)}$ and $\boldsymbol{\tau} \equiv\boldsymbol{\tau}^{(k)}$) and use tilded notation for the next iterate (e.g. $\wtd{\bp} \equiv \bp^{(k+1)}$). This convention extends to all variables.


  Since \(d(\bw)>0\), \(\gamma_{\beta} \ne 0\). For \(\forall \bt=[\bt_{11}^{\T},\dots,\bt_{st}^{\T}]^{\T} \in {\rm span}({\bf \Theta})\) with $\bt_{ij}\in \bbC^m$ partitioned as in \eqref{eq:a}, by Corollary \ref{coro:verticle_coro1}, we have 
  \begin{align}\nonumber
  0&=\langle \bt,\BF\bq-\bp\rangle_{\BW}\\\nonumber
  &=\bt^{\HH}\BW\left(\BF\bq-\bp\right)=\sum_{i=1}^s\sum_{j=1}^t \bt_{ij}^{\HH}W(F_{ij}\bq-\bp_{ij})
  \\\label{eq:twy}
  &=\sum_{\ell=1}^m w_\ell\left(\sum_{i=1}^s\sum_{j=1}^t (\be_\ell^{\T} \bar\bt_{ij}) (f_{ij}(x_\ell)q(x_\ell)-p_{ij}(x_\ell))\right).
  \end{align} 
According to Lawson's updating rule in Algorithm \ref{alg:mdLawson}, we have 
$$
\wtd \BW=\frac{1}{\gamma_\beta} (I_g \otimes\diag(\boldsymbol{\tau}(\beta))) \BW,~\wtd w_\ell =\frac{w_\ell \tau_\ell^{\beta}}{\gamma_{\beta}} ~(1\le \ell \le m).
$$  
By $\tau_\ell^2=\|F(x_\ell) - R(x_\ell)\|_{\F}^2 =\frac{1}{|q(x_\ell)|^2}\sum_{i=1}^s\sum_{j=1}^t|f_{ij}(x_\ell)q(x_\ell)-p_{ij}(x_\ell)|^2$  and $q(x_\ell)\ne 0$, a node $x_\ell$ satisfying $\tau_\ell = 0$ implies $f_{ij}(x_\ell)q(x_\ell)=p_{ij}(x_\ell)$ for all pairs $(i,j)$ with $1\le i\le s, 1\le j\le t$. Also, $w_\ell=0$ implies $\wtd w_\ell=0$. 
Thus, based on \eqref{eq:twy}, nodes $x_\ell$ with $w_\ell\tau_\ell=0$, i.e., $\ell \not\in {\cal J}_{\wtd \bw}$ in \eqref{eq:index_update}, satisfy $\wtd w_\ell=0$ and contribute nothing to $\langle \bt, \BF\bq - \bp \rangle_{\BW}$.
Consequently,  we obtain from \eqref{eq:twy} that
  \begin{align}\nonumber
      0&=\langle \bt,\BF\bq-\bp\rangle_{\BW}=\sum_{\ell\in {\cal J}_{\wtd \bw}}  w_\ell\left(\sum_{i=1}^s\sum_{j=1}^t (\be_\ell^{\T}\bar \bt_{ij}) (f_{ij}(x_\ell)q(x_\ell)-p_{ij}(x_\ell))\right)\\\nonumber
      &=\sum_{\ell\in {\cal J}_{\wtd \bw}}  \frac{\wtd w_\ell \gamma_\beta}{\tau_\ell^\beta}\left(\sum_{i=1}^s\sum_{j=1}^t (\be_\ell^{\T} \bar\bt_{ij}) (f_{ij}(x_\ell)q(x_\ell)-p_{ij}(x_\ell))\right)\\\nonumber
      &=\gamma_\beta\sum_{\ell\in {\cal J}_{\wtd \bw}} \wtd w_\ell\left(\sum_{i=1}^s\sum_{j=1}^t (\be_\ell^{\T} \bar\bt_{ij})  \frac{  f_{ij}(x_\ell)q(x_\ell)-p_{ij}(x_\ell)}{\tau_\ell^\beta}\right)
      \\\label{eq:0/0}
            &=\gamma_\beta\sum_{\ell=1}^m\wtd w_\ell\left(\sum_{i=1}^s\sum_{j=1}^t (\be_\ell^{\T} \bar\bt_{ij})  \frac{f_{ij}(x_\ell)q(x_\ell)-p_{ij}(x_\ell)}{\tau_\ell^\beta} \right) \\
    \nonumber
    &= \gamma_{\beta}\bt^{\HH}\wtd{\BW}\by=\gamma_{\beta}\langle\bt,\by\rangle_{\wtd{\BW}}.
  \end{align}
In deriving \eqref{eq:0/0}, we retain summations over nodes $x_\ell$ where $\wtd{w}_\ell = 0$ (i.e., $\ell \notin  {\cal J}_{\wtd{\bw}}$). These nodes correspond to either $w_\ell = 0$ or $\tau_\ell =|f_{ij}(x_\ell)q(x_\ell) - p_{ij}(x_\ell)| = 0$ for all $(i, j)$. For the latter scenario, we adopt the convention $0\frac{0}{0} = 0$ for terms of the form $(f_{ij}(x_\ell)q(x_\ell) - p_{ij}(x_\ell)) / \tau_\ell^\beta$.
\end{proof}
We next  provide a   relation between \(d(\bw^{(k)})\) and \(d(\bw^{(k+1)})\). 
\begin{theorem}\label{thm:dw_lower_bound}
Using the same notation and assumptions in Lemma \ref{lem:lem_verticle1}, it holds that
    \begin{equation}\label{eq:dw_lower_bound}
        \sqrt{d(\bw^{(k+1)})} \ge d(\bw^{(k)})\frac{\left|\left(\bq^{(k+1)}\right)^{\HH}W^{(k)}\bq^{(k)}\right|}{\gamma^{(k)}_{\beta}\zeta^{(k)}_{\beta}},
    \end{equation}
 where \(W^{(k)}={\rm diag}(\bw^{(k)})\), $\gamma^{(k)}_{\beta}$ is given in \eqref{eq:gammak} and 
    \begin{equation}\nonumber
        \zeta^{(k)}_{\beta}=\left\|\frac{\left(\BF\bq^{(k)}-\bp^{(k)}\right).}{\mathbf{1}_g \otimes \boldsymbol{\tau}^{(k)}(\beta)}\right\|_{\BW^{(k+1)}},~~{\BW}^{(k+1)}=I_g\otimes W^{(k+1)}.
    \end{equation}
\end{theorem}
\begin{proof}
We consider only the case where $d(\bw^{(k)})>0$, as  \eqref{eq:dw_lower_bound} holds trivially otherwise. 
Omit the subscript \(k\) in each quantity related with the \(k\)th iteration as  in the proof of Lemma \ref{lem:lem_verticle1}. 

Recall  \eqref{eq:index_update}: \({\cal J}_{\wtd{\bw}}=\{\ell: \wtd{{w}}_{\ell} \ne 0,~1\le \ell\le m\}=\{\ell:w_{\ell}\tau_{\ell} \ne 0,~1\le \ell\le m\}\).  Since \(d(\wtd{\bw})=\|\BF\wtd{\bq}-\wtd{\bp}\|_{\wtd{\BW}}^2\), and    \(\BF\wtd{\bq}-\wtd{\bp}\ \bot_{\wtd{\BW}}\ \wtd{\bp}\) (note $\wtd{\bp}\in {\rm span}({\bf \Theta})$) by Corollary \ref{coro:verticle_coro1}, we get
  \begin{equation}\nonumber
    0 \le d(\wtd{\bw})=\langle\BF\wtd{\bq}-\wtd{\bp},\BF\wtd{\bq}\rangle_{\wtd{\BW}}=\sqrt{d(\wtd{\bw})}\left\langle\frac{\BF\wtd{\bq}-\wtd{\bp}}{\|\BF\wtd{\bq}-\wtd{\bp}\|_{\wtd{\BW}}},\BF\wtd{\bq}\right\rangle_{\wtd{\BW}},
  \end{equation}
  which implies \(\sqrt{d(\wtd{\bw})}=\left\langle\frac{\BF\wtd{\bq}-\wtd{\bp}}{\|\BF\wtd{\bq}-\wtd{\bp}\|_{\wtd{\BW}}},\BF\wtd{\bq}\right\rangle_{\wtd{\BW}}\).

  Let $${\by}=\frac{\left(\BF\bq-\bp\right).}{\mathbf{1}_g \otimes \boldsymbol{\tau}(\beta)}=[\by_{11}^{\T},\dots,\by_{st}^{\T}]^{\T}  \mbox{ with } \by_{ij}=\frac{(F_{ij}\bq-\bp_{ij}).}{\boldsymbol{\tau}(\beta)}\in \bbC^{m}.$$ According to \eqref{eq:yperp}, we have $\by_{ij}\bot_{\wtd{\bw}}\ {\rm span}({\bf \Psi}_{ij})$.
    For any pair $(i,j)$, the optimality condition  $F_{ij}\wtd{\bq}-\wtd{\bp}_{ij}  \bot_{\wtd{\bw}}\ {\rm span}({\bf\Psi}_{ij})$  in Corollary \ref{coro:verticle_coro1} implies that $$\wtd{\bp}_{ij}=\argmin_{\bz\in {\rm span}({\bf \Psi}_{ij})}\|F_{ij}\wtd{\bq}-\bz\|_{\wtd\bw};$$  thus from $\by_{ij}\bot_{\wtd{\bw}}\ {\rm span}({\bf \Psi}_{ij})$, assumptions (\textbf{A1}), (\textbf{A2}) and \cite[Lemma 3.1]{zhha:2025}, it holds that 
    $$
    \left\langle F_{ij}\wtd{\bq},\frac{F_{ij}\wtd{\bq}-\wtd{\bp}_{ij}}{\|F_{ij}\wtd{\bq}-\wtd{\bp}_{ij}  \|_{\wtd{\bw}}}\right\rangle_{\wtd{\bw}}\ge \left|\left\langle F_{ij}\wtd{\bq},\frac{ {\by}_{ij}}{\| {\by}_{ij}\|_{\wtd{\bw}}}\right\rangle_{\wtd{\bw}}\right|,
    $$
    and hence
  \begin{align}\nonumber
    \sqrt{d(\wtd{\bw})} &=\left\langle\frac{\BF\wtd{\bq}-\wtd{\bp}}{\|\BF\wtd{\bq}-\wtd{\bp}\|_{\wtd{\BW}}},\BF\wtd{\bq}\right\rangle_{\wtd{\BW}}=\sum_{i=1}^s\sum_{j=1}^t \left\langle F_{ij}\wtd{\bq},\frac{F_{ij}\wtd{\bq}-\wtd{\bp}_{ij}}{\|F_{ij}\wtd{\bq}-\wtd{\bp}_{ij}  \|_{\wtd{\bw}}}\right\rangle_{\wtd{\bw}}\\\nonumber
    &\ge \sum_{i=1}^s\sum_{j=1}^t \left|\left\langle F_{ij}\wtd{\bq},\frac{ {\by}_{ij}}{\| {\by}_{ij}\|_{\wtd{\bw}}}\right\rangle_{\wtd{\bw}}\right| \ge \left|\sum_{i=1}^s\sum_{j=1}^t \left\langle F_{ij}\wtd{\bq},\frac{ {\by}_{ij}}{\| {\by}_{ij}\|_{\wtd{\bw}}}\right\rangle_{\wtd{\bw}}\right|\\
    &=\left|\left\langle\BF\wtd{\bq},\frac{ {\by}}{\| {\by}\|_{\wtd{\BW}}}\right\rangle_{\wtd{\BW}}\right|.\label{eq:squared}
  \end{align}
  Next, we   show \(\left|\langle\BF\wtd{\bq}, {\by}\rangle_{{\wtd\BW}}\right|=\frac{d(\bw)}{\gamma_{\beta}}\left|\wtd{\bq}^{\HH}W\bq\right|\), which together with \eqref{eq:squared} then leads to the desired result \eqref{eq:dw_lower_bound}. 
  Indeed, 
  \begin{align}\nonumber
    \left|\langle\BF\wtd{\bq}, {\by}\rangle_{\wtd{\BW}}\right|-\frac{d(\bw)}{\gamma_{\beta}}\left|\wtd{\bq}^{\HH}W\bq\right| &= \left|(\BF\wtd{\bq})^{\HH}\wtd{\BW}\frac{(\BF\bq-\bp).}{\mathbf{1}_g \otimes \boldsymbol{\tau}(\beta)}\right|-\frac{d(\bw)}{\gamma_{\beta}}\left|\wtd{\bq}^{\HH}W\bq\right| \\ &=\frac{1}{\gamma_{\beta}}\left(\left|\wtd{\bq}^{\HH}W\BF^{\HH}(\BF\bq-\bp)\right|-d(\bw)\left|\wtd{\bq}^{\HH}W\bq\right|\right).\nonumber
  \end{align}
Since \(\wtd{\bq} \in {\rm span}({\bf \Phi})\), by the second optimality \eqref{eq:verticle_coro1} in Corollary \ref{coro:verticle_coro1},   
  \begin{equation}\nonumber
    \BF^{\HH}(\BF\bq-\bp)-d(\bw)\bq\ \bot_{\bw}\ {\rm span}({\bf \Phi}) \Longrightarrow \wtd{\bq}^{\HH}W\left(\BF^{\HH}(\BF\bq-\bp)-d(\bw)\bq\right)=0,
  \end{equation}
  giving \(\left|\langle\BF\wtd{\bq}, {\by}\rangle_{\wtd{\BW}}\right|=\frac{d(\bw)}{\gamma_{\beta}}\left|\wtd{\bq}^{\HH}W\bq\right|\). By \eqref{eq:squared}, the proof is complete.
\end{proof}

\subsection{Convergence for matrix-valued polynomial minimax approximations}\label{subsec:beta1poly}
In this subsection, we   prove two interesting results of {\sf m-d-Lawson} for the polynomial case (i.e., \(d=0\)): (a)  $d(\bw^{(k+1)})\ge d(\bw^{(k)})$ for any $\beta\in [0,2]$, and (b) the choice $\beta=1$ is near-optimal in the sense that the lower bound in \eqref{eq:dw_lower_bound} attains its maximum when $\beta \ge  0$. 

For these properties, we note that when \(d=0\), then $q\equiv 1$ implying that \(\bq^{(k+1)} \equiv \bq^{(k)} \equiv \mathbf{1}_m \in \mathbb{R}^{m}\) and \(\left(\bq^{(k+1)}\right)^{\HH}W^{(k)}\bq^{(k)} \equiv 1\) for all \(k\). Thus the lower bound in \eqref{eq:dw_lower_bound} reduces to
\begin{equation}\label{eq:lower_bound_dw_poly}
    \sqrt{d(\bw^{(k+1)})} \ge d(\bw^{(k)})\frac{1}{\gamma^{(k)}_{\beta}\zeta^{(k)}_{\beta}}.
\end{equation}
Following  \cite[Section 4]{zhha:2025}, we   define a near-optimal Lawson exponent \(\beta_{*}^{(k)}\) as 
\begin{equation}\label{eq:nu_beta_minimization}
    \beta_{*}^{(k)}=\arg\min_{\beta \ge 0}\left(\gamma^{(k)}_{\beta}\zeta^{(k)}_{\beta}\right)^2.
\end{equation}
Omitting the subscript \(k\) in each quantity related with the \(k\)th iteration as in the proof of Lemma \ref{lem:lem_verticle1}, and recalling \(\tau_{\ell}=\|F(x_{\ell})-R(x_{\ell})\|_{\F}\) and \({\cal J}_{\wtd{\bw}}=\{\ell: \wtd{w}_{\ell} \ne  0, 1\le \ell\le m\}=\{\ell:{w}_{\ell}\tau_\ell \ne  0, 1\le \ell\le m\}\) in Proposition \ref{prop:index_update}, we obtain by $\wtd w_\ell =\frac{w_\ell \tau_\ell^\beta}{\gamma_\beta} ~(1\le \ell \le m)$ and $q\equiv 1$ that
\begin{align}\nonumber
    \nu(\beta)&:=\gamma_{\beta}^2\zeta_{\beta}^2 = \gamma_{\beta}^2 \sum_{i=1}^s\sum_{j=1}^t\left(\sum\limits_{\ell=1}^{m}\frac{\wtd{w}_{\ell}|f_{ij}(x_{\ell}) -p_{ij}(x_{\ell})|^2}{\tau_\ell^{2\beta}}\right) \\ \nonumber &= \gamma_{\beta} \sum\limits_{\ell=1}^{m} \frac{w_{\ell}}{\tau_\ell^{\beta}}\underbrace{\left(\sum_{i=1}^s\sum_{j=1}^t|f_{ij}(x_{\ell}) -p_{ij}(x_{\ell})|^2\right)}_{=\tau_\ell^2}
    \\
    \label{eq:nu_beta} &= \gamma_{\beta} \sum\limits_{\ell=1}^{m} \frac{w_{\ell}\tau_\ell^{2}}{\tau_\ell^{\beta}}    =  \underbrace{\left(\sum\limits_{\ell \in {\cal J}_{\wtd{\bw}}}w_\ell \tau_\ell^{\beta}\right)}_{=\gamma_\beta}\left(\sum\limits_{\ell \in {\cal J}_{\wtd{\bw}}}w_\ell   \tau_\ell^{2-\beta}\right).
\end{align}
Note that \(\nu(\beta)\) in \eqref{eq:nu_beta} takes the same expression form as in \cite{zhha:2025}; particularly $\nu(0)=\nu(2)=d(\bw^{(k)})$.  Therefore, we can straightforwardly generalize \cite[Proposition 4.2]{zhha:2025} to the case of matrix-valued polynomial minimax approximation:
\begin{proposition}\label{prop:converg_analysis_poly_lawson}
    For the matrix-valued polynomial minimax approximation, we have
    \begin{itemize}
        \item [(i)] the function \(\nu(\beta)\) given in \eqref{eq:nu_beta} is convex;
        \item [(ii)] \(\beta=1\) is the global minimizer of \eqref{eq:nu_beta_minimization}. In this sense, \(\beta=1\) achieves the maximum of the lower bound in \eqref{eq:lower_bound_dw_poly} and can be viewed as the near-optimal Lawson exponent for {\sf m-d-Lawson}  in Algorithm \ref{alg:mdLawson};
        \item [(iii)]   \(\forall \beta \in [0,2]\), the sequence $\{d(\bw^{(k)})\}_k$ of {\sf m-d-Lawson} satisfies \(d(\bw^{(k+1)}) \ge d(\bw^{(k)})\).
    \end{itemize}
\end{proposition}
\begin{proof}
With  \(\nu(\beta)\) given in \eqref{eq:nu_beta}, the proof is finished according to \cite[Section 4]{zhha:2025}. 
\end{proof}

\subsection{Convergence for matrix-valued rational minimax approximations}\label{subsec:convgrational}
We now analyze the convergence of   {\sf m-d-Lawson} (Algorithm \ref{alg:mdLawson})  for general rational minimax approximation. Unlike the polynomial case, a key challenge arises because the numerator \(\left|(\bq^{(k+1)})^{\HH}W^{(k)}\bq^{(k)}\right|\)  in the lower bound \eqref{eq:dw_lower_bound} also depends on the Lawson exponent $\beta$. However, owing to the structural similarities in the minimax framework, particularly the setting with a common denominator $q(x) \in \bbP_{d}$, the analysis in \cite[Section 5]{zhha:2025} can be extended to obtain the following result.

\begin{theorem}\label{thm:converg_mdl_gen}
    At the \(k\)th iteration  of  {\sf m-d-Lawson} in Algorithm \ref{alg:mdLawson}, for \(\bw^{(k)} \in {\cal S}\), assume \(d(\bw^{(k)})\) is a simple eigenvalue of the matrix pencil \((A_{\bw^{(k)}},B_{\bw^{(k)}})\) given in \eqref{eq:dual_GEP}.  Assume (\textbf{A1}) and (\textbf{A2}).  Suppose \(\tau^{(k)}_\ell:=\|F(x_{\ell})-R^{(k)}(x_\ell)\|_{\F}>0,\ \forall \ell \in {\cal J}_{\bw^{(k)}}:=\{j:w^{(k)}_{j} \ne 0, 1\le j\le m\}\). Then
    \begin{itemize}
        \item [(i)]   \(\exists\beta_{0}>0\) so that   \(\forall \beta \in (0,\beta_{0})\),  we have \(d(\bw^{(k+1)})\ge d(\bw^{(k)})\);
        \item [(ii)] for any sufficiently small \(\beta>0\),
        \begin{equation}\label{eq:ssssss}
            d(\bw^{(k+1)})=d(\bw^{(k)}) \Longrightarrow w^{(k)}_\ell \tau^{(k)}_\ell \left(\tau^{(k)}_\ell-c\right)=0,\ \forall 1 \le \ell \le m,
        \end{equation}
        where \(c=\sqrt{d(\bw^{(k)})} \le  \sqrt{e(R^{(k)})}=\max_{1\le \ell\le m}\tau^{(k)}_\ell\);
        \item [(iii)] for   item (ii), if additionally, \(c \ge \max_{\ell \not\in {\cal J}_{\bw^{(k)}}} \tau^{(k)}_\ell \), then \(R^{(k)}(x)\) is the minimax approximant of \eqref{eq:bestf0}.
    \end{itemize}
\end{theorem}
\begin{proof}
The proof adapts \cite[Theorem 5.2]{zhha:2025} with necessary adjustments. For completeness, we present the full argument here. The idea is to express and estimate the lower bound in \eqref{eq:dw_lower_bound} using the real parameter \(\beta\) around \(\beta=0\). For simplicity,  using the notation in the proof of Theorem \ref{thm:dw_lower_bound}, we consider the nontrivial case  \(d(\bw)>0\) (the case \(d(\bw)=0\) is trivial because $w_\ell \tau_\ell=0,~\forall 1\le \ell \le m$).

  Define \(W(\beta)={\rm diag}(w_1(\beta),\dots,w_m(\beta))\) with 
  \begin{equation}\nonumber
    w_j(\beta)=\frac{w_j\tau_j^{\beta}}{\gamma_{\beta}}=\frac{w_j\|F(x_j)-R(x_j)\|_{\F}^{\beta}}{\sum\limits_{\ell=1}^{m}w_{\ell}\|F(x_{\ell})-R(x_{\ell})\|_{\F}^{\beta}}
  \end{equation}
  for which we have \(w_{j}(0)=w_j\) and 
    $w_j^{\prime}(0)=w_j{\rm log}\tau_j-w_j\sum\limits_{\ell=1}^{m}w_{\ell}{\rm log}\tau_{\ell},\ \ \forall j \in {\cal J}_{\wtd{\bw}}.$

  In the following discussion, we can assume \(\bq(\beta)\) is continuously differentiable with respect to \(\beta\) around \(\beta=0\), where \(\bq(\beta)={\bf \Phi}\bb(\beta)\) and \(\ba(\beta) \in \mathbb{C}^{n}\), \(\bb(\beta)\in \mathbb{C}^{d+1}\) are the solution to \eqref{eq:Lagrangedualfun3} with \(W(\beta)\) (for more details, we refer to the proof of \cite[Theorem 5.2]{zhha:2025}). Hence,
   $\bq(\beta)=\bq(0)+\beta\bq^{\prime}(0)+O(\beta^2).$
  Since \(1 \equiv (\bq(\beta))^{\HH}W(\beta)\bq(\beta)\), differentiating it with respect to \(\beta\) and setting $\beta=0$ gives
  \begin{equation}\nonumber
    {\rm Re}((\bq^{\prime}(0))^{\HH}W(0)\bq(0))=-\frac{1}{2}(\bq(0))^{\HH}W^{\prime}(0)\bq(0);
  \end{equation}
  moreover, for any sufficiently small \(\beta>0\), we have
  \begin{align}\nonumber
    \left|(\bq(\beta))^{\HH}W(0)\bq(0)\right| & \ge \left|{\rm Re}((\bq(\beta))^{\HH}W(0)\bq(0))\right| \\ \nonumber &= \left|{\rm Re}(\left((\bq(0))^{\HH}+\beta \bq^{\prime}(0)^{\HH}+O(\beta^2)\right)W(0)\bq(0))\right| \\ \nonumber &= \left|(\bq(0))^{\HH}W(0)\bq(0)+\beta {\rm Re}((\bq^{\prime}(0))^{\HH}W(0)\bq(0))\right|+O(\beta^2) \\ \nonumber &= \left|1-\frac{\beta}{2}(\bq(0))^{\HH}W^{\prime}(0)\bq(0)\right|+O(\beta^2) \\ \nonumber &= 1-\frac{\beta}{2}\left(\sum\limits_{j \in {\cal J}_{\wtd{\bw}}}|q_j|^2w_j{\rm log}\tau_j-\sum\limits_{j \in {\cal J}_{\wtd{\bw}}}w_j{\rm log}\tau_j\right)+O(\beta^2),
  \end{align}
  where the last equality is due to \(\bq(0)=\bq\) and \(\sum\limits_{\ell=1}^{m}w_{\ell}|q_{\ell}|^2=1\). With this, we can write the lower bound in \eqref{eq:dw_lower_bound} as 
  \begin{align}\nonumber
    \what{\mu}(\beta) & := \frac{\left|(\bq(\beta))^{\HH}W(0)\bq(0)\right|}{\gamma_{\beta}\zeta_{\beta}} \\ \nonumber & \ge \frac{1-\frac{\beta}{2}\left(\sum\limits_{j \in {\cal J}_{\wtd{\bw}}}w_j|q_j|^2{\rm log}\tau_j-\sum\limits_{j \in {\cal J}_{\wtd{\bw}}}w_j{\rm log}\tau_j\right)}{\sqrt{\left(\sum\limits_{j \in {\cal J}_{\wtd{\bw}}}w_j\tau_j^{\beta}\right) \left(\sum\limits_{j \in {\cal J}_{\wtd{\bw}}}w_j|q_j|^2\tau_j^{2-\beta}\right)}}+O(\beta^2)   := \mu(\beta)+O(\beta^2)
  \end{align}
  locally at \(\beta=0\). Note $\what{\mu}(0)= \frac{1}{\sqrt{d(\bw(0))}}$. For \(\mu(\beta)\), by calculation, we have (with \(w_j(0)=w_j,\ \bq(0)=\bq\))
  \begin{align}\nonumber
    \mu^{\prime}(0)&=\frac{1}{2\sqrt{d(\bw)^3}}\left[\sum\limits_{j \in {\cal J}_{\wtd{\bw}}}w_j|q_j|^2\tau_j^2{\rm log}\tau_j-\left(\sum\limits_{j \in {\cal J}_{\wtd{\bw}}}w_j|q_j|^2\tau_j^2\right) \left(\sum\limits_{j \in {\cal J}_{\wtd{\bw}}}w_j|q_j|^2{\rm log}\tau_j\right)\right] \ge 0
  \end{align}
  where the last inequality is obtained by applying \cite[Lemma 5.1]{zhha:2025}. Moreover, \cite[Lemma 5.1]{zhha:2025} also says that if there is a pair \((i,j)\) such that \(i,j\in {\cal J}_{\wtd{\bw}}\) and \(w_iw_j(\tau_i-\tau_j)\ne 0\), then \(\mu^{\prime}(0)>0\). In that case, we know that there is a \(\beta_0>0\) such that \(\mu^{\prime}(\beta)>\frac{1}{2}\mu^{\prime}(0)\) and \(\frac{\beta}{2}\mu^{\prime}(0)+O(\beta^2)>0\) for any \(\beta \in (0,\beta_0)\), implying  $$\what{\mu}(\beta)>\what{\mu}(0)= \frac{1}{\sqrt{d(\bw(0))}},~\forall \beta\in (0,\beta_0).$$
  This shows that if  \(\exists(i,j)\) so that \(i,j \in {\cal J}_{\wtd{\bw}}\) and \(w_iw_j(\tau_i-\tau_j) \ne 0\), then a sufficiently small \(\beta>0\) leads to \(d(\wtd{\bw})=d(\bw(\beta)) \ge d(\bw(0))=d(\bw)\); conversely, for a sufficiently small \(\beta>0\),
  \begin{align}\nonumber
    d(\bw(\beta))=d(\bw(0)) & \Longrightarrow w_iw_j(\tau_i-\tau_j)=0,\ \forall i,j \in {\cal J}_{\wtd{\bw}}\\ & \Longrightarrow \tau_j={\rm a \ constant}\ c,\ \forall j \in {\cal J}_{\wtd{\bw}}.
  \end{align}
  Obviously, \(c \le \max_{1 \le \ell \le m}\tau_\ell=\sqrt{e(R)}\) and also
  \begin{equation}
    d(\bw)=\sum\limits_{j=1}^mw_j|q_j|^2\tau_j^2=\sum\limits_{j \in {\cal J}_{\wtd{\bw}}}w_j|q_j|^2\tau_j^2=c^2\sum\limits_{j \in {\cal J}_{\wtd{\bw}}}w_j|q_j|^2=c^2\sum\limits_{j=1}^mw_j|q_j|^2=c^2.
  \end{equation}
  Hence \(\sqrt{d(\bw)}=c \le \sqrt{e(R)}\),   proving items (i) and (ii).

  For item (iii), if we additionally have \(c \ge \max_{\ell \not\in {\cal J}_{\bw}}\tau_\ell\), then \(c=\sqrt{d(\bw)}=\sqrt{e(R)}\). According to Theorem \ref{thm:strongdualityRuttan}, we conclude that \(R^{(k)}(x)\) is the minimax approximant of \eqref{eq:bestf0}. In this case,   \eqref{eq:ssssss} can be written as 
  $  w_\ell\left(\tau_\ell-\sqrt{e(R)}\right)=0,\ \forall 1\le  \ell \le m,$
the complementary slackness property given in Theorem \ref{thm:complement}.
\end{proof}

\section{Numerical experiments}\label{sec:numerical}
To verify the numerical performance of the method {\sf m-d-Lawson}\footnote{See \url{https://ww2.mathworks.cn/matlabcentral/fileexchange/181757-m_d_lawson} for the MATLAB code of  {\sf m-d-Lawson}.} in Algorithm \ref{alg:mdLawson}, we implement it in MATLAB R2024b in double precision and carry out numerical experiments on a 14-inch MacBook Pro with an M3 Pro chip and 18Gb memory. 
The numerical examples are from \cite{gogu:2021,trai:2010,trms:2007,zhzy:2025}, which have been used to test the following methods for matrix-valued rational approximations: {\sf v-AAA-Lawson}\footnote{See \url{https://ww2.mathworks.cn/matlabcentral/fileexchange/174740-v-aaa-lawson} for the MATLAB code of {\sf v-AAA-Lawson}.} \cite{zhzy:2025}, the  block-AAA \cite{gogu:2021}, RKFIT \cite{begu:2017},  Loewner framework interpolation approach \cite{anli:2017,gova:2022,maan:2007} and the vector fitting (VF) approach   in \cite{gust:2009}. Codes of all the latter four   are available at \url{https://github.com/nla-group/block_aaa}.  

\begin{example}\label{eg:Eg1}
The first test problem is a toy example from \cite{gogu:2021}. The target is a symmetric matrix-valued rational function $F:\bbC\rightarrow \bbC^{2\times 2}$  given by 
$$F(x)=\left[\begin{array}{cc}\frac{2}{x+1} & \frac{3-x}{x^2+x-5} \\ \frac{3-x}{x^2+x-5} & \frac{2+x^2}{x^3+3x^2-1}\end{array}\right].
$$
It can be seen  that all entries $f_{ij}$ of $F$ can  be written as type $(5,6)$ rational functions $f_{ij}=p_{ij}/q$ with the denominator $q(x)=(x + 1)(x^2 + x -  5)(x^3 + 3x^2-1)\in \bbP_6$.   
To get {sample data  from $F$},  we use {$m  = 1000$ equally spaced samples} of $F$ in the interval $[1, 100]{\tt i}$ in the imaginary axis  to have $\{x_\ell,F(x_\ell)\}_{\ell=1}^{m}$.  To see the numerical performance of the tested methods, we also introduce a noise term $\sigma_{\ell}$ for each  $f_{ij}(x_{\ell})$ at node $x_{\ell}$ 
\[ \what f_{ij}(x_{\ell})  = f_{ij}(x_{\ell})+ \sigma_{\ell}, \quad i,j = 1,2, ~\mbox{and}~\ell=1,\dots,m,\]
where the real and imaginary part of $\sigma_{\ell}$ are drawn from a normal distribution with mean zero  and standard deviation (noise level) $\varsigma.$ 
\end{example}

We remark that  the current implementations of block-AAA\footnote{We notice that the order  $d$ of the denominator specified in block-AAA for approximating a matrix-valued function $F:\bbC\rightarrow \bbC^{n\times n}$ generally produces rational approximations of type $(dn,dn)$ rational functions $r_{ij}$. Therefore, in our case with $q\in \bbP_6$, we specify the  order  $3$  in block-AAA.}, RKFIT\footnote{RKFIT is an iterative method, and RKFIT($k$) refers to RKFIT with $k$ iterations.},  Loewner approach and VF provide rational approximations $r_{ij}$ of type $(d,d)$ for $f_{ij}$, while  {\sf m-d-Lawson} and  {\sf v-AAA-Lawson} are able to compute  $r_{ij}$ with user-defined types. Therefore, we set the exact types $(5,6)$ for the latter two, while choose type $(6,6)$ and default settings for others.   

For a computed approximation, accuracy is evaluated using both the root mean squared error (RMSE) in \cite{gogu:2021} given by \eqref{eq:RMSE}
and the maximum error 
$
e(R)=\max_{x_{\ell}\in {\cal X}}\|F(x_{\ell})-R(x_{\ell})\|_{\F}^2.
$
In Table \ref{Tab:toy1-a}, we report the average {\tt RMSE}  and $e(R)$ over 10 random tests. Note {\sf m-d-Lawson}(10) and {\sf v-AAA-Lawson}(10) mean that the maximal number of Lawson iterations is $\ell_{\max}=10$ for  {\sf m-d-Lawson} and {\sf v-AAA-Lawson}, respectively.  
 
\begin{table}[h!!!]
\caption{The average {\tt RMSE}  and $\sqrt{e(R)}$ over 10 random tests for Example \ref{eg:Eg1}}
\begin{center}\resizebox{150mm}{26mm}{\tabcolsep0.08in
\begin{tabular}{|c|c|c|c|c|c|c|}\hline
Method &{\sf m-d-Lawson}(10)  &{\sf v-AAA-Lawson}(10) &  block-AAA  &  RKFIT(10)  &  Loewner & VF\\
 {\tt RMSE} & type (5,6)  & type (5,6)  &   type (6,6)   &  type (6,6) &   type (6,6)&   type (6,6) \\
\hline
\hline
$   \varsigma = 0 $   &        1.9455e-14  &        4.6607e-15   &      9.3332e-14 &      4.8719e-13 &      4.0952e-03 &      7.1585e-15  \\  
 $   \varsigma = 10^{-10}  $   &        3.2919e-10  &        9.0328e-10   &      4.6273e-07 &      8.8925e-09 &      5.7620e-03 &      8.9246e-09  \\  
 $   \varsigma = 10^{-8}  $   &        3.3406e-08  &        8.2984e-08   &      1.7531e-05 &      8.8809e-07 &      5.1548e-03 &      8.9133e-07  \\  
 $   \varsigma = 10^{-6}  $   &        6.0453e-06  &        1.0248e-05   &      7.8522e-04 &      8.9351e-05 &      5.4559e-03 &      8.9615e-05  \\  
 $   \varsigma = 10^{-4}  $   &        3.4093e-04  &        3.4481e-04   &      3.9847e-01 &      8.9501e-03 &      9.8599e-02 &      9.6356e-03  \\  
 \hline\hline
 Method &{\sf m-d-Lawson}(10)  &{\sf v-AAA-Lawson}(10) &  block-AAA  &  RKFIT(10)  &  Loewner & VF\\
$\sqrt{e(R)}$ & type (5,6)  & type (5,6)  &   type (6,6)   &  type (6,6) &   type (6,6)&   type (6,6) \\
\hline
$   \varsigma = 0 $   &        2.8788e-14  &        2.7243e-14   &      1.7939e-14 &      8.8405e-14 &      3.3193e-03 &      2.5466e-15  \\  
 $   \varsigma = 10^{-10}  $   &        6.4018e-10  &        2.5945e-08   &      1.5342e-07 &      5.2593e-10 &      4.9532e-03 &      5.2486e-10  \\  
 $   \varsigma = 10^{-8}  $   &        6.6641e-08  &        2.3500e-06   &      1.0511e-05 &      5.1277e-08 &      4.2331e-03 &      5.1317e-08  \\  
 $   \varsigma = 10^{-6}  $   &        1.6581e-05  &        2.8447e-04   &      1.1752e-04 &      5.0985e-06 &      4.6080e-03 &      5.1365e-06  \\  
 $   \varsigma = 10^{-4}  $   &        9.1874e-04  &        2.6519e-03   &      2.8424e-01 &      5.2663e-04 &      7.2616e-02 &      1.3503e-03  \\  
 \hline
\end{tabular}
}
\end{center}
\label{Tab:toy1-a}
\end{table}
  
\begin{example}\label{eg:Eg2}
Our second test problem is a symmetric matrix-valued  function $F:\bbC\rightarrow \bbC^{2\times 2}$ from \cite{gogu:2021}  
$$F(x)=\left[\begin{array}{cc}\frac{x(1-2x\cot 2x)}{\tan x-x} +10& \frac{x(2x-\sin 2x)}{\sin (2x) (\tan x -x)} \\ \frac{x(2x-\sin 2x)}{\sin (2x) (\tan x -x)}& \frac{x(1-2x\cot 2x)}{\tan x-x} +4\end{array}\right].
$$
This function is a 2-by-2 submatrix of the buckling example in \cite{hint:2019}, arising from the nonlinear eigenvalue problem in a  buckling plate model. As each $f_{ij}(x)$ is not rational functions, we approximate $f_{ij}$ 
 by a type $(d,d)$ rational function $r_{ij}$ with $d=10$ and  with a common denominator.
In the same settings for the noise as Example \ref{eg:Eg1}, we  choose $m=500$ logarithmically spaced sampling points in the interval $[10^{-2}, 10]{\tt i}$, and report the accuracy of each method in Table \ref{Tab:toy2-a}.   
\end{example}

\begin{table}[h!!!]
\caption{The average {\tt RMSE}  and $\sqrt{e(R)}$ over 10 random tests for Example \ref{eg:Eg2} with $d=10$}
\begin{center}\resizebox{150mm}{26mm}{\tabcolsep0.08in
\begin{tabular}{|c|c|c|c|c|c|c|}\hline
Method &{\sf m-d-Lawson}(10)  &{\sf v-AAA-Lawson}(10) &  block-AAA  &  RKFIT(10)  &  Loewner & VF\\
 {\tt RMSE} & type (10,10)  & type (10,10)  &   type (10,10)   &  type (10,10) &   type (10,10)&   type (10,10) \\
\hline
\hline
 $   \varsigma =0$   &        4.2986e-10  &        7.6848e-10   &      3.1982e-09 &      6.5385e-09 &      2.0619e-03 &      3.3010e-05  \\  
 $   \varsigma = 10^{-10}  $   &        5.0496e-10  &        7.7655e-10   &      2.9644e-07 &      9.0328e-09 &      2.6409e-03 &      3.3010e-05  \\  
 $   \varsigma = 10^{-8}  $   &        3.3155e-08  &        1.0942e-07   &      1.6510e-05 &      6.2993e-07 &      5.0303e-03 &      3.3015e-05  \\  
 $   \varsigma = 10^{-6}  $   &        7.0065e-06  &        1.0586e-05   &      1.1819e-03 &      6.2831e-05 &      4.0850e-01 &      7.1128e-05  \\  
 $   \varsigma = 10^{-4}  $   &        4.7199e-04  &        8.9167e-04   &      1.4120e-01 &      6.5554e-03 &      1.4752e+02 &      6.4432e-03  \\  
 \hline\hline
 Method &{\sf m-d-Lawson}(10)  &{\sf v-AAA-Lawson}(10) &  block-AAA  &  RKFIT(10)  &  Loewner & VF\\
 $\sqrt{e(R)}$ & type (10,10)  & type (10,10)  &   type (10,10)   &  type (10,10) &   type (10,10)&   type (10,10)  \\
\hline
 $   \varsigma = 0  $   &        6.3915e-10  &        1.4332e-09   &      2.8392e-10 &      1.7800e-09 &      1.3051e-03 &      1.0515e-05  \\  
 $   \varsigma = 10^{-10}  $   &        8.7366e-10  &        1.7199e-09   &      1.7874e-07 &      1.7818e-09 &      1.6258e-03 &      1.0515e-05  \\  
 $   \varsigma = 10^{-8}  $   &        1.2094e-07  &        2.4175e-07   &      5.1816e-06 &      5.0084e-08 &      2.9495e-03 &      1.0514e-05  \\  
 $   \varsigma = 10^{-6}  $   &        3.4276e-05  &        2.6732e-05   &      5.1681e-04 &      5.0782e-06 &      2.2459e-01 &      1.0911e-05  \\  
 $   \varsigma = 10^{-4}  $   &        2.1204e-03  &        2.5433e-03   &      7.2350e-02 &      5.8388e-04 &      7.4088e+01 &      5.3505e-04  \\  
 \hline
\end{tabular}
}
\end{center}
\label{Tab:toy2-a}
\end{table}

With this example, we next demonstrate the property of complementary slackness \eqref{eq:complement} tied
to optimality at extremal points. To this end, for type (6,6),  in Figure \ref{Fig:exam2_rat66errs}, we plot the error $e_{ij}(x_{\ell}) = |f_{ij}(x_{\ell}) - r_{ij}(x_{\ell})|$ for the individual $r_{ij}$, the F-norm error $ \| F(x_{\ell}) - R(x_{\ell})\|_{\F}$  as well as the 
sequences of $\left\{\sqrt{e(R^{(k)})}\right\}_k$ and  $\left\{\sqrt{d(\boldsymbol{w}^{(k)})}\right\}_k$. Noteworthy observations from the bottom-right subfigure of Figure \ref{Fig:exam2_rat66errs} are that $\left\{\sqrt{d(\bw^{(k)}}\right\}_k$ increases monotonically and the duality gap $\sqrt{e(R^{(k)})}-\sqrt{d(\bw^{(k)})}$ converges to zero as the iteration count $k$ increases. Furthermore, the top-right subfigure in Figure \ref{Fig:exam2_rat66errs} reveals another  interesting observation: exactly 11 extreme points attain the maximum approximation error $e(R)$ and
the complementary slackness \eqref{eq:complement} property is demonstrated numerically as 
$$
\max_{1\le \ell\le m}\left|w_\ell (e(R)-\|F(x_\ell)-R(x_\ell)\|_{\F}^2)\right|=6.6680\times 10^{-14}.
$$
This aligns with a key property of minimax approximation which is related to the famous equioscillation property \cite[Theorem 24.1]{tref:2019a} for the real scalar case. In contrast, such extreme points are absent in the other methods compared in Figure \ref{Fig:exam2_rat66compares}.

\begin{figure}[h!!!]
	\centering
	\includegraphics[width=0.97\columnwidth, height = 5.7cm]{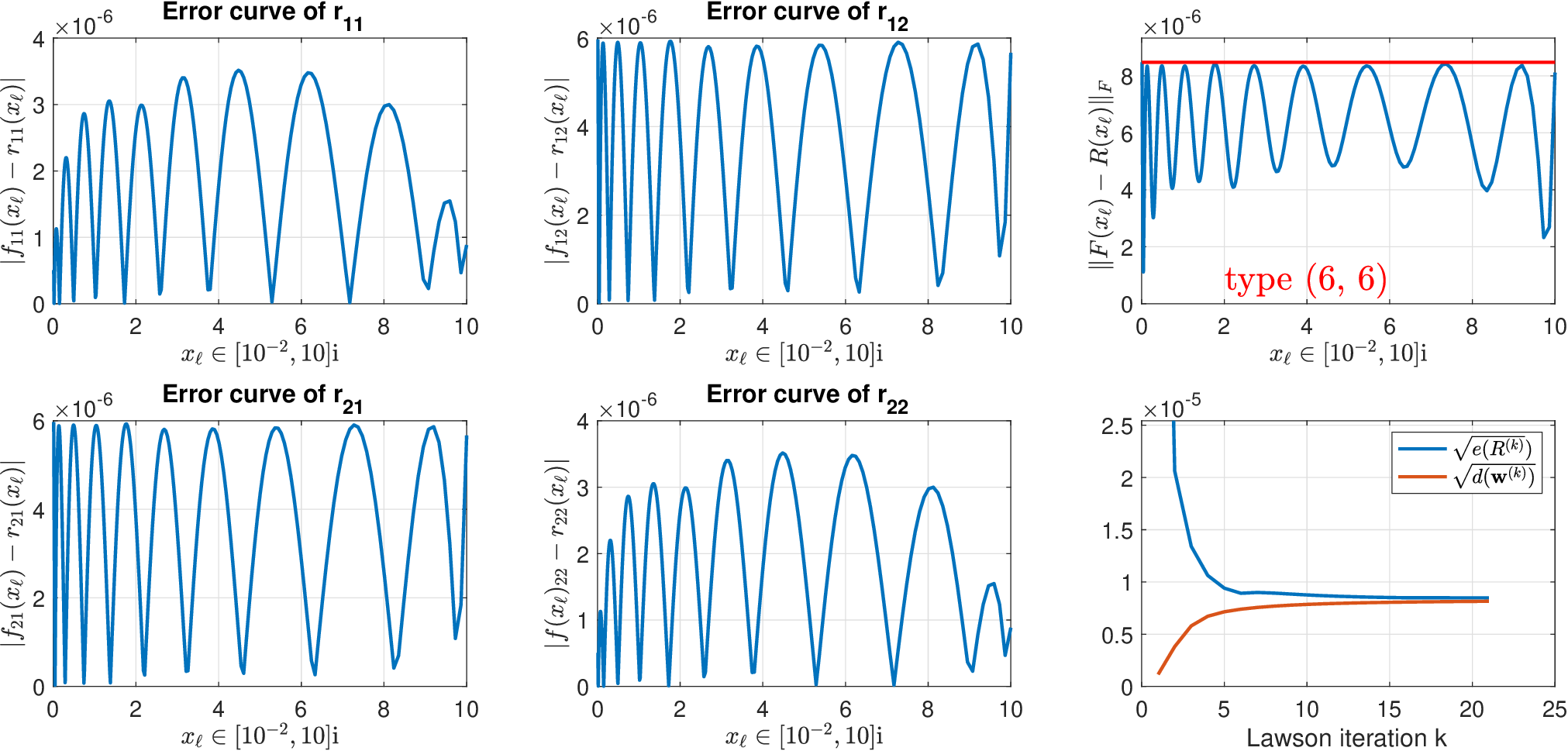}
	\caption{\small The errors of the approximation $R(x) $ with type (6,6) computed by {\tt m-d-Lawson(20)} for Example \ref{eg:Eg2}. The left two columns show the 
	absolute errors $e_{ij}(x_{\ell}) = |f_{ij}(x_{\ell}) -r_{ij}(x_{\ell})|$ of $r_{ij}(x),$  the up-right   represents the F-norm error $ \| F(x_{\ell}) - R(x_{\ell})\|_{\F}$  at each $x_\ell$ and 
	the bottom-right  demonstrates that the duality gap $\sqrt{e(R^{(k)})}-\sqrt{d(\bw^{(k)})}$ decreases as  $k$ increases. }
\label{Fig:exam2_rat66errs}\end{figure}

\begin{figure}[h!!!]
	 \centering
	\includegraphics[width=0.97\columnwidth, height = 5.7cm]{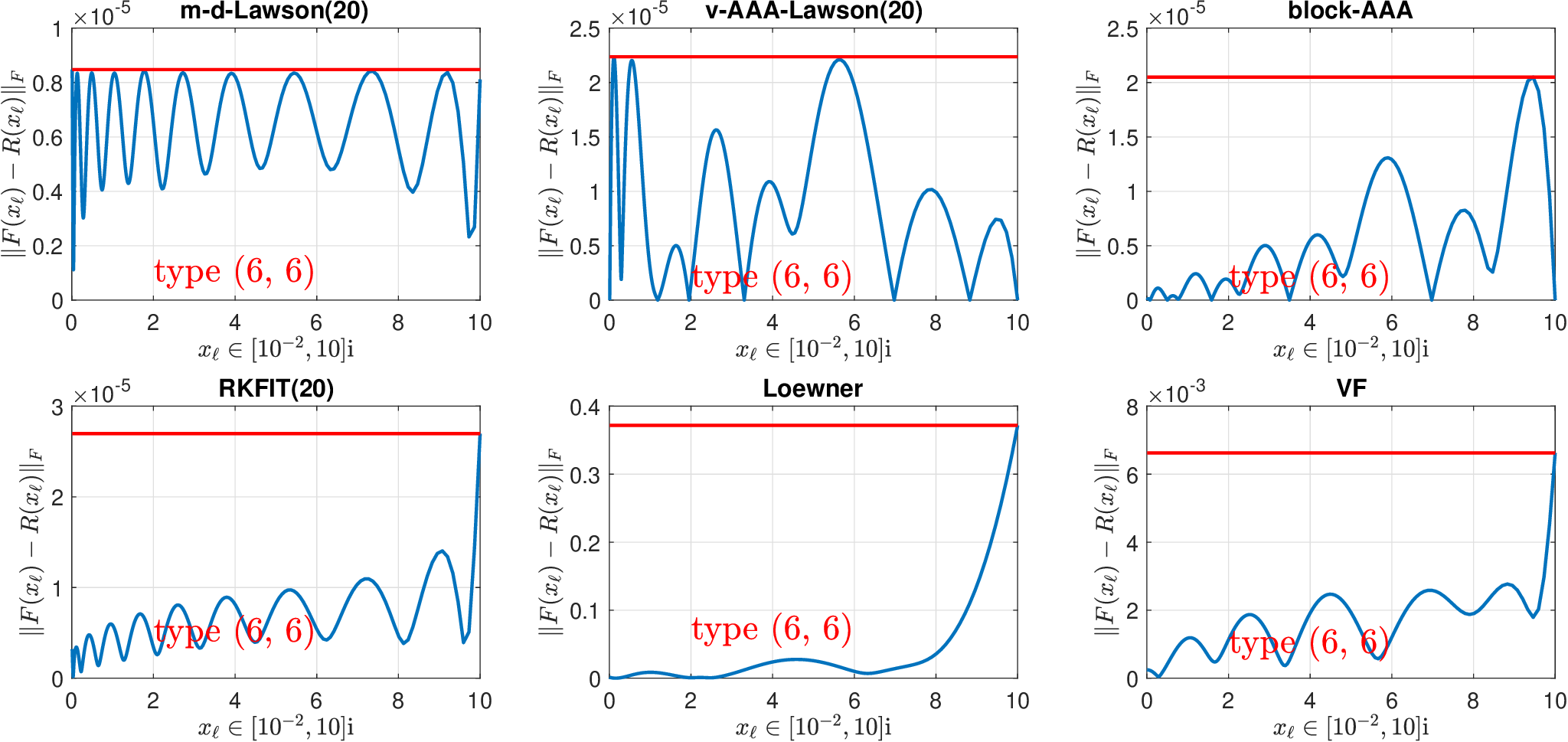} 
	\caption{\small The F-norm errors $ \| F(x_{\ell}) - R(x_{\ell})\|_{\F}$  of the approximation $R(x) $ with type (6,6) computed by different methods for Example \ref{eg:Eg2}. }
\label{Fig:exam2_rat66compares}\end{figure}

We also use this example to demonstrate the performance of {\tt m-d-Lawson} for computing the best matrix-valued  {\it polynomial} approximation of $F(x)$. In particular, we set the degree $d = 0$ for the denominator  and $n_{ij}=12$ for all numerators;  thus $R(x)$ becomes a matrix-valued  polynomial
\[ P(x) = 
\begin{bmatrix}
p_{11}(x) & p_{12}(x) \\
p_{21}(x) & p_{22}(x)
\end{bmatrix}:\bbC\rightarrow \bbC^{2\times 2}, ~p_{ij}\in \bbP_{12}.
\] 
In this case, similarly, we report   the error $e_{ij}(x_{\ell}) = |f_{ij}(x_{\ell}) - p_{ij}(x_{\ell})|$,  $ \| F(x_{\ell}) - P(x_{\ell})\|_{\F}$  as well as the sequences of $\left\{ \sqrt{e(P^{(k)})}\right\}_k$ and  $\left\{\sqrt{d(\boldsymbol{w}^{(k)})}\right\}_k$ in Figure \ref{Fig:exam2_polyerrs}. Comparing with Figure \ref{Fig:exam2_rat66errs}, we observe that the polynomial    and  rational cases share a similar performance with 
$$
\max_{1\le \ell\le m}\left|w_\ell (e(R)-\|F(x_\ell)-R(x_\ell)\|_{\F}^2)\right|=2.2900\times 10^{-11}.
$$


\begin{figure}[h!!!]
	\centering
	\includegraphics[width=0.97\columnwidth, height = 5.7cm]{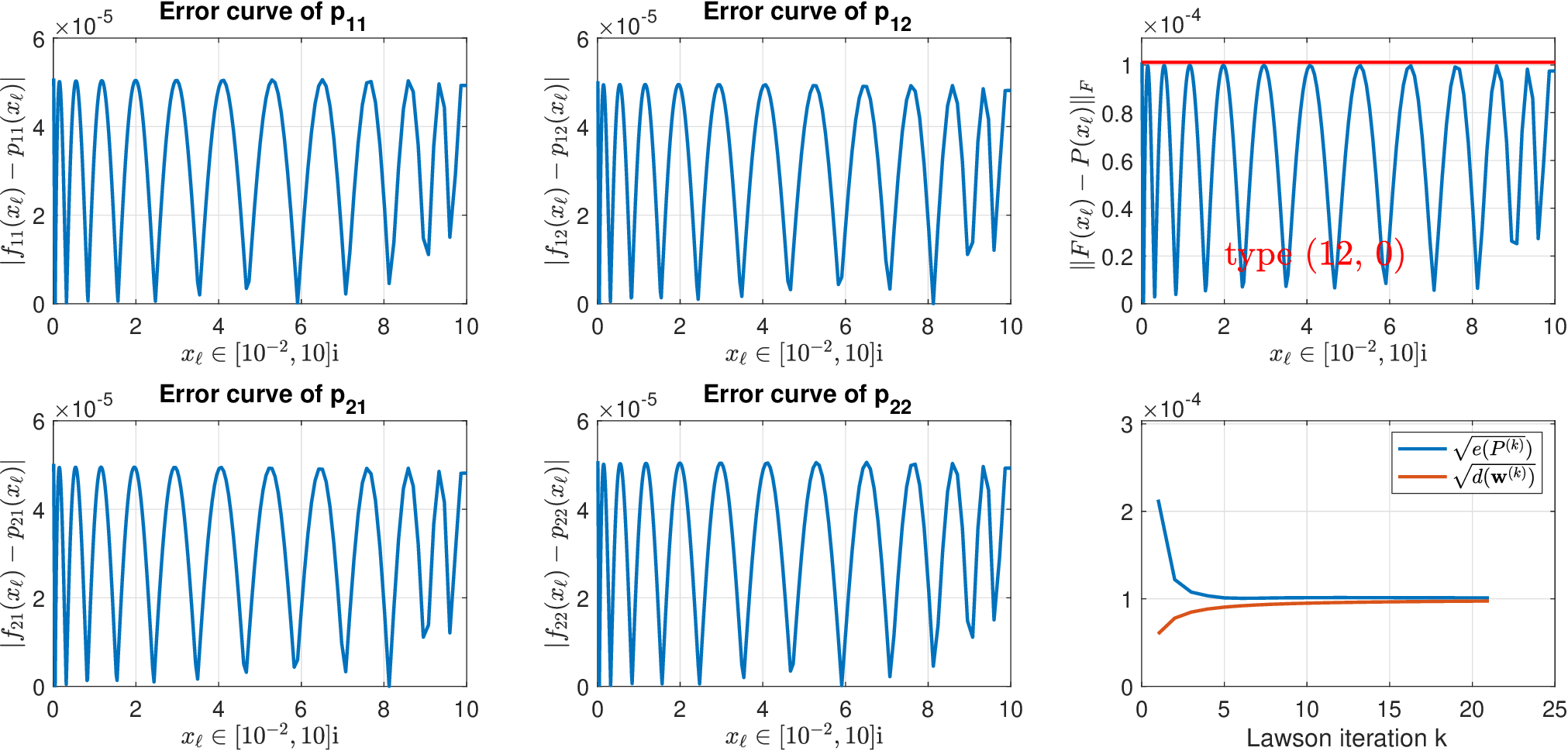}
	\caption{\small The errors of the polynomial approximation $P(x) $ with type (12,0) computed by {\tt m-d-Lawson(20)} for Example \ref{eg:Eg2}. The left two columns show the 
	absolute errors $e_{ij}(x_{\ell}) = |f_{ij}(x_{\ell}) -p_{ij}(x_{\ell})|$ of $p_{ij}(x),$  the up-right   represents the F-norm error $ \| F(x_{\ell}) - P(x_{\ell})\|_{\F}$  at each $x_\ell$, and the bottom-right  demonstrates that the duality gap $\sqrt{e(P^{(k)})}-\sqrt{d(\bw^{(k)})}$ decreases as  $k$ increases. }
\label{Fig:exam2_polyerrs}\end{figure}

A final illustration with Example \ref{eg:Eg2} is to demonstrate the errors of {\tt RMSE} and $e(R)$ with various types other than $(d,d)$.  In particular, we apply   {\sf v-AAA-Lawson}(20) and {\sf m-d-Lawson}(20) to two cases: $(n_{ij}, 20)$  and $(n_{ij},n_{ij}+2)$  with $n_{ij}=1,\dots,20$. Figure \ref{Fig:toy3-a} illustrates {\tt RMSE} and $e(R)$ for these cases, which implies that {\sf m-d-Lawson}(20) generally produces more accurate approximation than {\sf v-AAA-Lawson}(20). 

\begin{figure}[h!!!]
\centering
	\includegraphics[width=0.95\columnwidth, height = 7.5cm]{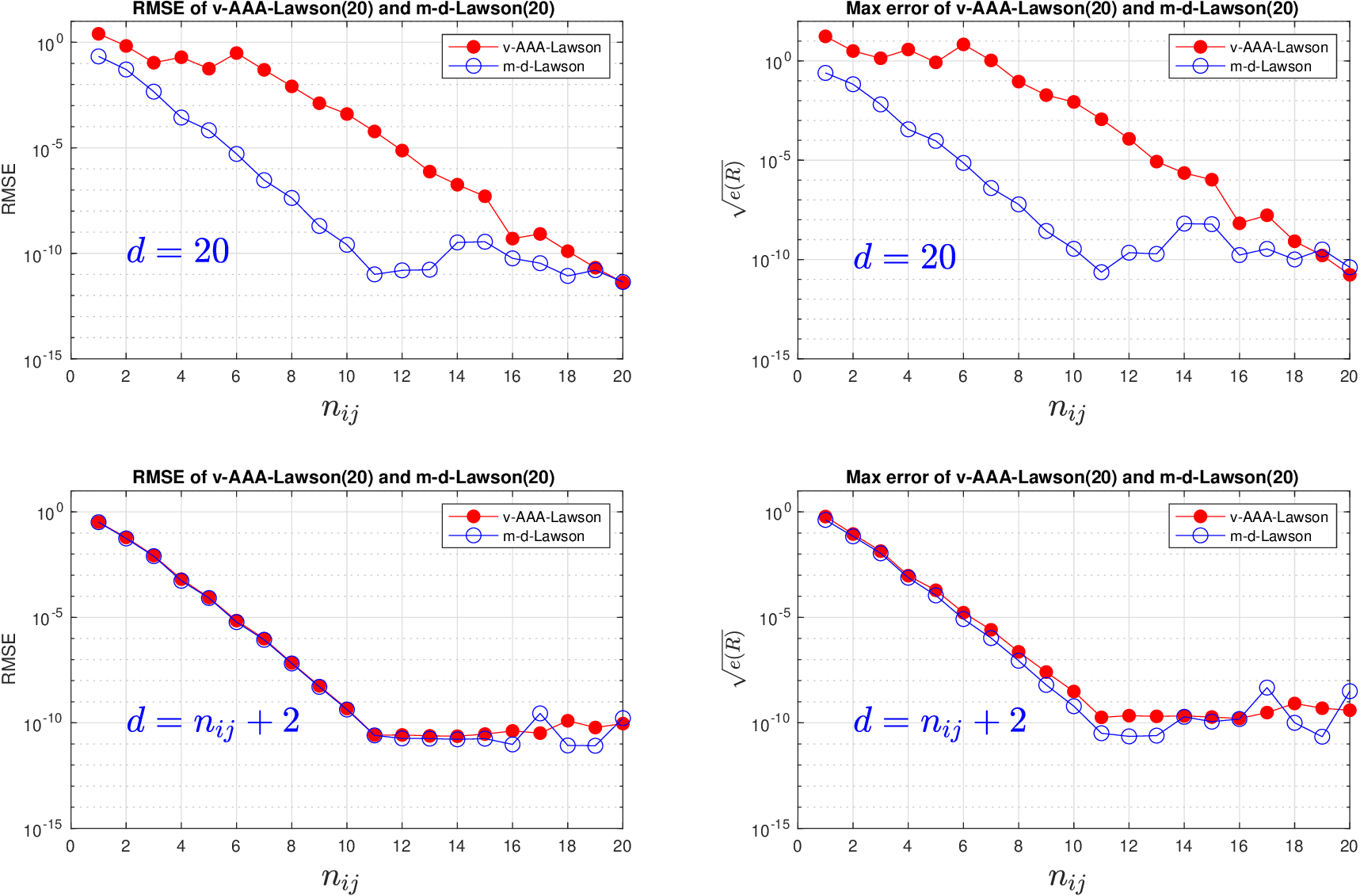}
	\caption{\small {\tt RMSE} (left column) and $\sqrt{e(R)}$ (right column) from  {\sf v-AAA-Lawson}(20) and {\sf m-d-Lawson}(20)  based on rational approximation types $(n_{ij}, 20)$ and $(n_{ij}+2, n_{ij})$ with $n_{ij}=2,\dots,20$ for Example \ref{eg:Eg2}.}
\label{Fig:toy3-a}\end{figure}

\begin{example}\label{eg:Eg3}
We now test the case arising from the designed duplexer device with given poles and zeros in \cite{trai:2010,trms:2007}. In this example, three rational functions, i.e., scattering parameters, 
\begin{equation}\label{eq:S}
s_{11}(x)=\frac{p_N(x)}{p_{D}(x)}, ~~s_{21}(x)=\frac{p_{T}(x)}{p_{D}(x)},~\mbox{and}~~ s_{31}(x)=\frac{p_{R}(x)}{p_{D}(x)},
\end{equation}
need to be reconstructed from sampled responses measured at a set of frequencies.
Here, $p_{N}\in \bbP_{20},~p_{D}\in \bbP_{20},~p_T\in \bbP_{12},~p_R\in \bbP_{12}$ are complex polynomials, and the sample nodes are from the line $x={\tt i}\omega$ with $\omega\in \bbR$ being the frequency. In our application, there are $401$ (i.e., $m=400$) equidistant sampled  frequencies ${\cal X}=\{(-2+\frac{\ell}{400}){\tt i}\}_{\ell=0}^{400}$ as well as responses $\{s_{j1}(x_\ell)\}_{\ell=0}^{400}$ for $j=1,2,3$.  The zeros of associated  polynomials of $p_{D}(x),~p_{N}(x),~p_T(x)$ and $p_R(x)$ are given in Table \ref{Tab:polzer}. 
\end{example}

To apply {\sf v-AAA-Lawson} and {\sf m-d-Lawson}, we set $s=3$ and $t=1$ to form a vector-valued complex function $F(x)=[s_{11}(x), s_{21}(x), s_{31}(x)]^{\T}: \bbC\rightarrow \bbC^{3}$.  In order to verify the  robustness, we put a random perturbation to the responses as follows
\[ \what s_{j1}(x_\ell)  = s_{j1}(x_\ell)+ \sigma_{\ell}, \quad j = 1,2,3, ~\mbox{and}~\ell=0,\dots,m,\]
where $\sigma_{\ell}$ is the noise with noise level $\varsigma$ described  in Example \ref{eg:Eg1}. 
In the right column of Figure \ref{Fig:cauchy_recon-b}, we plot the fitting curves $y_j(x)=20\log_{10}|s_{j1}(x)|~(j=1,2,3)$  as well as their rational reconstructions computed  from {\sf v-AAA-Lawson}(10) and {\sf m-d-Lawson}(10) for the noise level  $\varsigma=10^{-8}$; the left two columns plot the roots (solid points) given in Table \ref{Tab:polzer} and the computed ones (circles) from {\sf v-AAA-Lawson}(10) and {\sf m-d-Lawson}(10).  
\begin{table}[h!!!]
\caption{\small Roots of the associated polynomials of a duplexer from \cite{trai:2010,trms:2007}}
\begin{center}
\resizebox{125mm}{35mm}{\tabcolsep0.2in
\begin{tabular}{|c|c|c|c|}\hline
 \text{roots of} $p_{D}(x)$ &  \text{roots of} $p_{N}(x)$ &  \text{roots of} $p_{T}(x)$ &  \text{roots of} $p_{R}(x)$ \\
\hline
  -0.0396-1.0564{\tt i} & -0.0050-0.0173{\tt i} &  0.0130-0.0070{\tt i} & -0.0071+0.1128{\tt i}\\
-0.1099-0.9971{\tt i} &  0.0404+0.0037{\tt i} & -0.0048+0.0181{\tt i} & -0.0228+0.1335{\tt i}\\
-0.1452-0.8985{\tt i} & -0.0221-0.0921{\tt i} & -0.0309-0.0016{\tt i} &  0.0063+0.1481{\tt i}\\
-0.0349+1.0418{\tt i} & -0.0009+0.1261{\tt i} & -0.0098-0.0443{\tt i} & -0.0461+0.1752{\tt i}\\
-0.0990+0.9864{\tt i} & -0.0151+0.1542{\tt i} & -0.0501-0.0915{\tt i} & -0.0068+0.2229{\tt i}\\
-0.1481-0.7199{\tt i} & -0.0217-0.2031{\tt i} & -0.0647-0.2039{\tt i} & -0.0740+0.2899{\tt i}\\
-0.1404+0.8811{\tt i} &  0.0231+0.2331{\tt i} & -0.0994-0.3527{\tt i} & -0.0735+0.4291{\tt i}\\
-0.1637+0.7315{\tt i} & -0.0223+0.2925{\tt i} & -0.0890-0.5222{\tt i} & -0.0954+0.5850{\tt i}\\
-0.1497-0.5206{\tt i} & -0.0621-0.3582{\tt i} & -0.0858-0.7178{\tt i} & -0.0809+0.7525{\tt i}\\
-0.1668+0.5697{\tt i} &  0.0310+0.4292{\tt i} & -0.0650-0.9007{\tt i} & -0.0765+0.8988{\tt i}\\
-0.1388-0.3422{\tt i} & -0.0142-0.5192{\tt i} & -0.0527-1.0230{\tt i} & -0.0430+1.0096{\tt i}\\
-0.1484+0.4140{\tt i} & -0.0309+0.5772{\tt i} & -0.0072-1.1030{\tt i} & -0.0053+1.1038{\tt i}\\
\hline
-0.1170+0.2846{\tt i} & -0.0249-0.7174{\tt i} & & \\
-0.0999-0.1895{\tt i} &  0.0352+0.7185{\tt i} & & \\
-0.0767+0.1975{\tt i} & -0.0281+0.8746{\tt i} & & \\
-0.0386+0.1414{\tt i} &  0.0049-0.9350{\tt i} & & \\
-0.0111+0.1166{\tt i} &  0.0957-0.9407{\tt i} & & \\
-0.0636-0.0772{\tt i} &  0.0497+0.9520{\tt i} & & \\
-0.0427+0.0014{\tt i} & -0.0163+1.0071{\tt i} & & \\
-0.0200-0.0024{\tt i} & -0.0299-1.0232{\tt i} & & \\
\hline
\end{tabular}\label{Tab:polzer}
}
\end{center}
\label{table02}
\end{table}%

\begin{figure}[h!!!]
	\centering
	\subfigure{ 
		\includegraphics[width=0.96\columnwidth, height = 6.2cm]{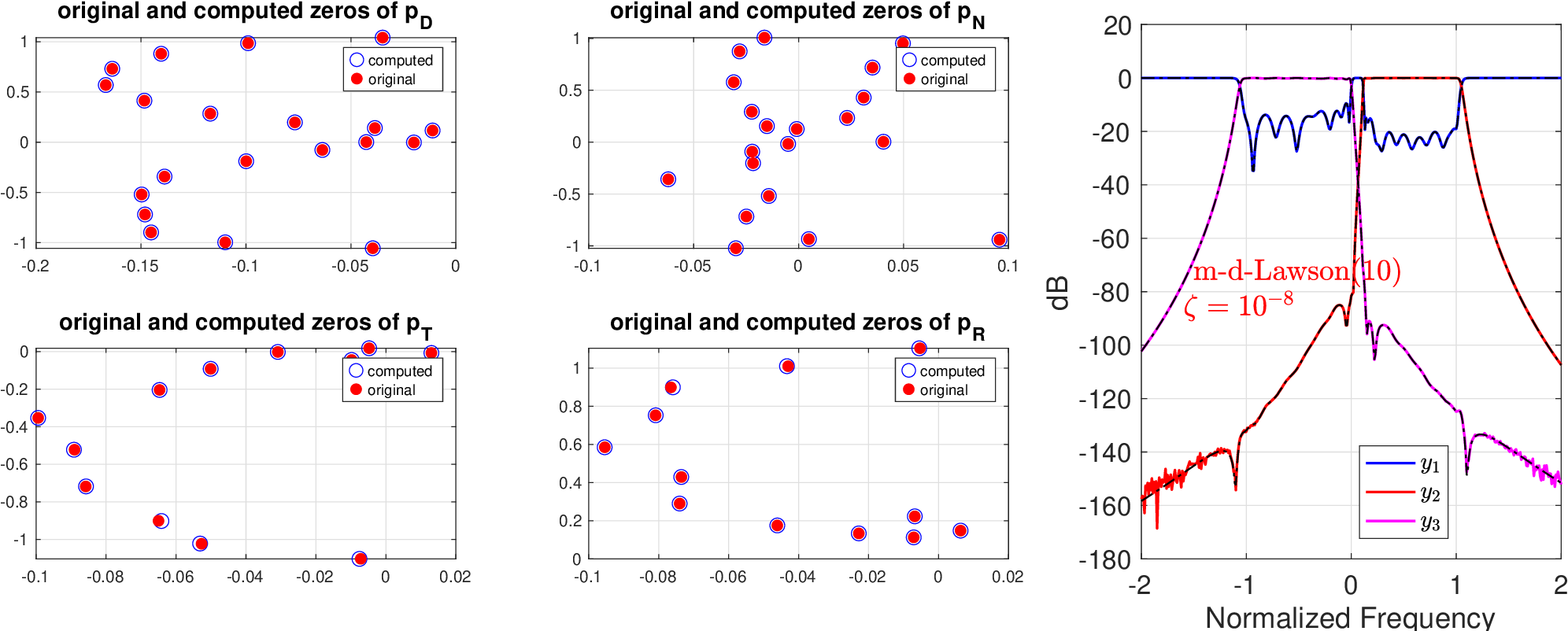}
	} \\
		 \subfigure{ 
		\includegraphics[width=0.96\columnwidth, height = 6.2cm]{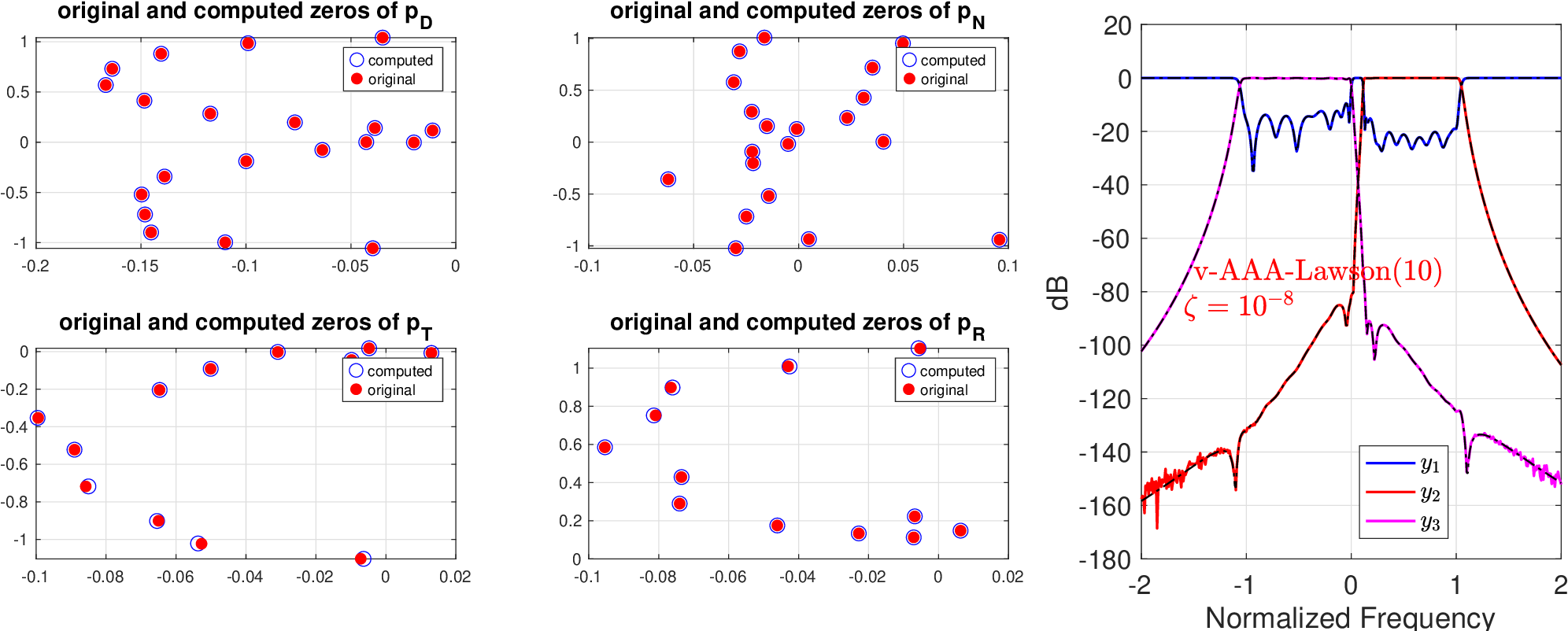}
	}
	\caption{\small Synthesized  $y_j(x)$ (thick line in the right column) for $j=1,2,3$  and fitted responses (thin dashed line  in the right column)  from  {\sf m-d-Lawson}(10) (top) and {\sf v-AAA-Lawson}(10) (bottom) with noise level $\varsigma=10^{-8}$.  The left two columns plot the roots (solid points) given in Table \ref{Tab:polzer} and the computed ones (circles).}
\label{Fig:cauchy_recon-b}\end{figure}

\section{Conclusions}\label{sec:conclusion}
Rational approximation of matrix-valued functions plays an important role in various applications, including nonlinear eigenvalue problems \cite{guti:2017}. Although several numerical methods based on discrete point sets have been proposed, most lack a well-defined optimization objective, resulting in approximations without guaranteed optimality.

In this paper, inspired by classical minimax approximation for scalar functions, we rigorously formulate a matrix-valued rational minimax approximation problem \eqref{eq:bestf0}. Exploiting recent advances in duality-based minimax approximation \cite{yazz:2023,zhha:2025,zhyy:2025}, we reformulate \eqref{eq:bestf0} as a max-min dual problem, which yields several key advantages:
\begin{itemize}
\item Computational tractability: The dual problem reduces to a real-valued optimization over the probability simplex, allowing us to extend the {\sf d-Lawson} iteration \cite{zhyy:2025} to the proposed {\sf m-d-Lawson} method.
\item Convergence guarantees: We establish the convergence of {\sf m-d-Lawson} for both matrix-valued polynomial and rational approximations, rigorously analyzing the influence of the Lawson exponent $\beta$.
\item Theoretical insights: Our duality analysis provides a sufficient condition (Theorem \ref{thm:strongdualityRuttan}) ensuring that the dual solution coincides with the best approximant for \eqref{eq:bestf0}. Furthermore, we characterize optimal solutions via complementary slackness for extreme points and strong duality.
\end{itemize}
Numerical experiments demonstrate that {\sf  m-d-Lawson} can  solve \eqref{eq:bestf0} efficiently, offering a new computational framework for matrix-valued rational approximation.

{\small
\def\noopsort#1{}\def\l{\char32l}\def\v#1{{\accent20 #1}}
  \let\^^_=\v\def\hbk{hardback}\def\pbk{paperback}
\providecommand{\href}[2]{#2}
\providecommand{\arxiv}[1]{\href{http://arxiv.org/abs/#1}{arXiv:#1}}
\providecommand{\url}[1]{\texttt{#1}}
\providecommand{\urlprefix}{URL }

}
     

%
\end{document}